\documentclass[12pt]{amsart}

\usepackage{amssymb}
\usepackage{latexsym,color}
\usepackage{tikz}
\usetikzlibrary{arrows,shapes,trees}

\usepackage{verbatim,euscript}
\usepackage{fullpage,array}

\usepackage[colorlinks=true,linkcolor=blue,urlcolor=my_color,citecolor=magenta]{hyperref}

\definecolor{my_color}{rgb}{0,0.5,0.5}
\definecolor{MIXT}{rgb}{0.4,0.3,0.6}
\definecolor{mixt}{rgb}{0.5,0.3,0.2}
\definecolor{sin}{rgb}{0,0.5,0.5}
\definecolor{darkblue}{rgb}{0,0.1,0.8}
\definecolor{redi}{rgb}{0.5,0,0.4}

\numberwithin{equation}{section}

\newtheorem{thm}{Theorem}[section]
\newtheorem{prop}[thm]{Proposition}
\newtheorem{lm}[thm]{Lemma}
\newtheorem{cl}[thm]{Corollary}
\newtheorem{conj}[thm]{Conjecture}
\newtheorem{qtn}{Question}
\newtheorem{prob}{Problem}

\theoremstyle{remark}
\newtheorem{rmk}[thm]{Remark}

\theoremstyle{definition}
\newtheorem{df}{Definition}
\newtheorem{ex}[thm]{Example}

\newcommand {\eus}{\EuScript}

\newcommand {\be}{{\mathfrak b}}

\newcommand {\g}{{\mathfrak g}}
\newcommand {\h}{{\mathfrak h}}

\newcommand {\el}{{\mathfrak l}}

\newcommand {\p}{{\mathfrak p}}
\newcommand {\q}{{\mathfrak q}}
\newcommand {\rr}{{\mathfrak r}}
\newcommand {\es}{{\mathfrak s}}
\newcommand {\te}{{\mathfrak t}}

\newcommand {\z}{{\mathfrak z}}

\newcommand {\fX}{{\mathfrak X}}

\newcommand {\gl}{\mathfrak{gl}}
\newcommand {\sel}{\mathfrak{sl}}

\newcommand {\gln}{\mathfrak{gl}_n}

\newcommand {\sln}{\mathfrak{sl}_n}

\newcommand {\glv}{\mathfrak {gl}(\VV)}

\newcommand {\spn}{\mathfrak {sp}_{2n}}

\newcommand {\gC}{{\eus C}}

\newcommand {\gK}{{\eus K}}
\newcommand {\gM}{{\eus M}}

\newcommand {\gS}{{\eus S}}
\newcommand {\gZ}{{\eus Z}}
\newcommand {\gV}{{\eus V}}

\newcommand {\ap}{\alpha}

\newcommand {\lb}{\lambda}

\newcommand {\cb}{{\mathcal B}}

\newcommand {\N}{{\mathcal N}}
\newcommand {\co}{{\mathcal O}}

\newcommand {\cs}{{\mathcal S}}

\newcommand {\md}{/\!\!/}

\newcommand {\ad}{{\mathrm{ad\,}}}

\newcommand {\codim}{{\mathrm{codim}}}

\newcommand {\Hom}{{\mathsf{Hom}}}
\newcommand {\ind}{{\mathrm{ind\,}}}
\newcommand {\Lie}{{\mathrm{Lie\,}}}
\newcommand {\Ker}{{\mathsf{Ker}}}
\newcommand {\Ima}{{\mathsf{Im}}}

\newcommand {\rk}{{\mathrm{rk\,}}}
\newcommand {\spe}{{\mathrm{Spec\,}}}

\newcommand {\trdeg}{{\mathrm{trdeg\,}}}

\newcommand {\tri}{{\mathfrak{sl}}_2}

\newcommand {\GR}[2]{{\textrm{{\color{redi}\bf #1}}}_{#2}}

\newcommand {\BA}{{\mathbb A}}

\newcommand {\BN}{{\mathbb N}}

\newcommand {\BP}{{\mathbb P}}

\newcommand {\VV}{{\mathbb V}}
\newcommand {\UU}{{\mathbb U}}
\newcommand {\WW}{{\mathbb W}}

\newcommand {\blb}{\boldsymbol{\lb}}

\newcommand {\ov}{\overline}

\newcommand {\beq}{\begin{equation}}
\newcommand {\eeq}{\end{equation}}

\renewcommand{\le}{\leqslant}
\renewcommand{\ge}{\geqslant}
\renewcommand{\lg}{\langle}
\newcommand{\rg}{\rangle}
\newcommand {\bbk}{\Bbbk}
\newcommand{\wrt}{{w.r.t.}}

\begin{document}
\setlength{\parskip}{2pt plus 3pt minus 0pt}
\hfill {\scriptsize May 27, 2026} 
\vskip1ex 

\title[Coadjoint orbits and the index of subalgebras]{The index of subalgebras and strange coadjoint orbits}
\author[D.\,Panyushev]{Dmitri I. Panyushev}
\address{Independent University of Moscow, 119002 Moscow, Russia}
\begin{abstract}
For an algebraic group $Q$ with $\Lie Q=\q$, we develop a method for estimating the index of a 
subalgebra $\h\subset\q$ via the use of coadjoint $Q$-orbits in $\q^*$. Let $\q^\xi$ denote the stabiliser of
$\xi\in\q^*$. In the special case when $\q^\xi\oplus\h=\q$, our estimate implies that $\ind\h=0$. Using our 
theory, we also answer a question of Duflo. An orbit $Q{\cdot}\eta\subset\q^*$ is said to be {\it strange}, if 
$\q^\eta\oplus\h=\q$ for some $\h$. In the second part of the paper, we study strange orbits for a 
semisimple algebra $\g$. It is shown that an orbit $\co\subset\g\simeq\g^*$ is strange whenever the 
complexity of $\co$ is at most 1. Furthermore, if $\cs\subset\g$ is a sheet containing a strange nilpotent 
orbit, then all orbits in $\cs$ are strange. We also show that strange orbits in $\sln$ are not as sparse, as 
one might expect, and discuss some conjectures on strange orbits.
\end{abstract}
\maketitle

\section{Introduction}     
\label{sect:intro}
Let $Q$ be an algebraic group with $\Lie Q=\q$. The {\it index\/} of $\q$, denoted $\ind\q$, is an important 
numerical invariant related to the coadjoint representation of $\q$. Let $\cb_\xi$ be the {\it Kirillov form\/} 
associated with $\xi\in\q^*$. It is an alternate bilinear form on $\q$ that is defined by 
$\cb_\xi(x,y)=\xi([x,y])$. Let $\q^\xi$ be the kernel of $\cb_\xi$. It is also the stabiliser of $\xi$ \wrt\ the 
coadjoint representation of $\q$ and $\ind\q=\min_{\xi\in\q^*}\dim\q^\xi$. The set of regular elements of 
$\q^*$ is $\q_{\sf reg}^*=\{\xi\mid \dim\q^\xi=\ind\q\}$. If $\q$ is reductive, then $\ind\q=\rk\q$. But
it is a non-trivial task to compute the index of non-reductive Lie algebras; especially, 
because the behaviour of index for a pair $\h\subset\q$ is unpredictable.

In Section~\ref{sect:estimate-index}, we provide a method for estimating the index of subalgebras of $\q$ 
using coadjoint $Q$-orbits in $\q^*$.  Given $\xi\in\q^*$ and a subalgebra $\h\subset\q$, set 
$a=\codim (\q^\xi+\h)$ and $b=\dim (\q^\xi\cap\h)$. Under natural mild assumptions, we prove that 
$\ind\h\le a+b$. In particular, if $a+b\le 1$, then $\ind\h=a+b$. The most interesting case occurs if
$a=b=0$, i.e., $\q^\xi\oplus\h=\q$. Then $\ind\h=0$, $\dim\h=\dim Q{\cdot}\xi$ and the group $H$ has a 
dense orbit in $Q{\cdot}\xi$. 

Using this theory, we answer a question of M.\,Duflo. Let $G$ be a semisimple group with 
$\Lie G=\g$ and $\h=\Lie H$ a subalgebra of $\g$ that does not contain proper ideals of $\g$. Let 
$p:\g\simeq\g^*\to \h^*$ be the natural projection. In Section~\ref{sect:quest}, we prove that 
$p(G{\cdot}x)$ is dense in $\h^*$ for any $x\in\g_{\sf reg}$. Therefore, 
$p(G{\cdot}x)\cap\h^*_{\sf reg}\ne\varnothing$, which is the affirmative answer to Duflo's question. The 
proof also exploit some invariant-theoretic properties of the adjoint representation of $G$. In particular,
we prove an auxiliary result, which is of interest in its own right. Namely, if $H\subset G$ is 
reductive, then the $H$-action on the nilpotent cone $\N\subset\g$ is stable and locally free,
see Appendix~\ref{app}.

Following M.\,Ra\"\i s, a coadjoint orbit $Q{\cdot}\xi\subset\q^*$ is said to be {\it strange}, if there is a 
subalgebra $\h$ such that $\q^\xi\oplus\h=\q$. Then $\dim\h=\dim Q{\cdot}\xi$ and $\h$ must be 
Frobenius, i.e., $\ind\h=0$. We also say that $\h$ is a {\it complementary subalgebra\/} for $Q{\cdot}\xi$.
In the main part of this article, we study strange orbits in semisimple Lie 
algebras. Long ago I proved that the spherical $G$-orbits in $\g$ are strange, and 
here much more results are presented. It is shown that the nilpotent orbit of complexity 1 is strange
(such orbit is only present in $\sln$, $n\ge 3$).
We also prove that if a nilpotent orbit in a sheet of $\g$ is strange,
then all orbits in this sheet are strange (cf. Section~\ref{sect:strange}).

The very name suggested by Ra\"\i s hints that ``strange'' orbits are rare and enigmatic. But being truly 
mysterious, these orbits are not uncommon in $\sln$. Recall that the nilpotent orbits in $\sln$ are 
parametrised by the partitions of $n$. If $\blb$ is a partition, then $\co(\blb)$ denotes the corresponding 
orbit. In Sections~\ref{sect:str-glv} and \ref{sect:two-col}, we provide several classes of strange orbits: 
\begin{enumerate}
\item if $\tilde e\in\sln$ is regular nilpotent, then the $SL_n$-orbit of $\tilde e^k$ is strange for any 
$k\ge 1$;
\item if $\blb$ contains only two parts and $\co(\blb)$ is strange, then either $\co(\blb)=SL_n{\cdot}\tilde e^2$ or $\blb=(m{+}1,m{-}1)$ for $n=2m$ and $\blb=(m{+}2,m{-}1)$ for $n=2m{+}1$;
\item  we point out three series of strange orbits such that $\blb$ has three parts. These are 
$\co(m{+}2,m{-}1,m{-}1)$, $\co(m{+}3,m{-}1,m{-}1)$, and $\co(m{+}3,m,m{-}1)$.
\end{enumerate}
An interesting  feature of partitions given in (2),(3) is that the corresponding strange orbits admit a 
{\bf parabolic} complementary subalgebra. This also happens in many other cases, e.g. $\co(3,2,1)$. 
These results show that all nilpotent orbits in $\sel_5$ are strange (hence all orbits are strange), while 
$\co(5,1)$ is the only non-strange nilpotent orbit in $\sel_6$. Having done a lot of computations for $\sln$ 
(not all of them are presented in the text), I got an impression that the presence of strange non-spherical orbits in 
$\sln$ is closely related to the fact that there are many Frobenius parabolic subalgebras. On the other hand, if 
$\be$ is a Borel subalgebra of a simple Lie algebra $\g$ and $\ind\be=0$, then $\be$ is the only 
Frobenius parabolic subalgebra. My conjecture is that in this case only spherical orbits are strange. We 
discuss this and other related conjectures in Section~\ref{sect:problems}.

{\it\bfseries Some notation.} The ground field $\bbk$ is algebraically closed and $\text{char\,}\bbk=0$. 
Algebraic groups are denoted by capital Latin letters and their Lie algebras are denoted by the corresponding small Gothic letters, e.g.
$\be=\Lie B$.
If an algebraic group $Q$ acts on an algebraic variety $X$, then this is denoted by $(Q:X)$.
For a given action $\ap: Q\times X \to X$, we usually write $s{\cdot}x$ in place of $\ap(s,x)$ for
$s\in Q$ and $x\in X$. An action $(Q{:}X)$ is said to be {\it locally free}, if there is a dense open subset 
$\Omega\subset X$ such that all the stabilisers $Q^x$, $x\in\Omega$, are finite. Then we also say that 
generic stabilisers for $(Q{:}X)$ are finite. Write $S^0$ for the identity component of an algebraic group $S$.

\section{An estimate of the index of a subalgebra}     
\label{sect:estimate-index}
 
Let $Q$ be a connected algebraic group with $\q=\Lie Q$. For $\xi\in\q^*$, the $Q$-orbit of $\xi$ with 
respect to the coadjoint representation of $Q$ is denoted by $Q{\cdot}\xi$. Then $Q^\xi$ is the 
{\it stabiliser\/} of $\xi$ in $Q$ and $\q^\xi=\Lie Q^\xi$. Let $\q^*_{\sf reg}$
denote the set of {\it regular elements} of $\q^*$, i.e.,
\[
       \q^*_{\sf reg}=\{\xi\in\q^*\mid \dim\q^\xi \text{ is minimal}\}=\{\xi\in\q^*\mid \dim Q{\cdot}\xi 
       \text{ is maximal}\}.
\]
Hence $\q^*_{\sf reg}$ is a dense open subset of $\q^*$. By definition, the {\it index\/} of $\q$, $\ind\q$, 
equals $\dim\q^\xi$ for some (any) $\xi\in\q^{\sf reg}$. Let $\bbk(\q^*)^Q$ denote the field of $Q$-invariant 
rational functions on $\q^*$. By the Rosenlicht theorem~\cite[\S\,2.3]{vp}, 
$\trdeg \bbk(\q^*)^Q=\dim\q- \max_{\xi\in\q^*} \dim Q{\cdot}\xi$. Hence one also has
$\ind\q=\trdeg \bbk(\q^*)^Q$. Since the coadjoint orbits are even-dimensional, one has
$\dim\q\equiv \ind\q\pmod 2$.

Let $H$ be a connected algebraic subgroup of $Q$ with $\h=\Lie H$. We assume that $\h\ne \{0\}$ and 
consider the projection $p:\q^*\to \h^*$. 
If $M$ is any subset of $\q$, then $M^\perp$ stands for its {\it annihilator\/} in $\q^*$, i.e., 
$M^\perp=\{\eta\in\q^*\mid \eta(m)=0 \text{ for all } m\in M\}$.  The linear span of $M$ is denoted by 
$\lg M\rg$.
Recall that $\q{\cdot}\xi=\mathsf T_\xi(Q{\cdot}\xi)\subset\q^*$ is the tangent space of the orbit 
$Q{\cdot}\xi$ at $\xi$ and $\q{\cdot}\xi=(\q^\xi)^\perp$.

\begin{lm}   \label{lm:1}
For any $\h\subset\q$ and $\xi\in\q^*$, we have
\begin{itemize}
\item[\sf (i)] \  $\dim(\q^\xi+\h)^\perp=\dim \q{\cdot}\xi-\dim\h{\cdot}\xi$;
\item[\sf (ii)] \ $\dim p(\q{\cdot}\xi)=\dim \h{\cdot}\xi=\dim\h-\dim (\h\cap\q^\xi)$.
\end{itemize}
\end{lm}
\begin{proof}
{\sf (i)}  $\dim(\q^\xi+\h)^\perp=\dim\q-\dim(\q^\xi+\h)=
\dim\q-\dim\q^\xi-\dim\h+\dim(\q^\xi\cap\h)$.
\\
{\sf (ii)}  We have $\dim p(\q{\cdot}\xi)=\dim(\q{\cdot}\xi)-\dim (\q{\cdot}\xi\cap\h^\perp)$. Since
$\q{\cdot}\xi\cap\h^\perp=(\q^\xi+\h)^\perp$, the assertion follows from {\sf (i)}.
\end{proof}

Set $\co=Q{\cdot}\xi$.  For $\eta\in\q^*$, we set $\bar\eta=p(\eta)=\eta\vert_{\h}$.
\begin{lm}    \label{lm:2}
{\sf 1.} If the action $(H:\co)$ is non-trivial, then $\Omega:=\{\eta\in\co\mid \bar\eta\ne 0\}$ is a dense 
open subset of $\co$. \\
{\sf 2.}  Moreover, if there is $\xi\in\co$ such that $\q^\xi+\h=\q$ and $\q^\xi\ne \q$, then $[\h,\h]\ne 0$ 
and $\bar\xi\ne 0$.
\end{lm}
\begin{proof}
{\sf 1.}  It is clear that  $\Omega$ is an open subset of $\co$. If $\bar\eta=0$ for all $\eta\in\co$, 
then $\lg Q{\cdot}\eta\rg \subset\h^\perp$. Hence 
$\q{\cdot}\eta=\mathsf T_\eta(Q{\cdot}\eta)\subset\h^\perp$. Then $\q^\eta\supset\h$ and the $H$-action 
on $\co$ is trivial.

{\sf 2.} If $\xi\vert_{[\q,\q]}=0$, then $\q^\xi=q$. Hence $\xi\vert_{[\q,\q]}\ne 0$. Next, 
$[\q,\q]=[\q,\q^\xi +\h]=[\q,\q^\xi]+ [\h,\h]$. Since $\xi\vert_{[\q,\q^\xi]}=0$, we obtain
$\xi\vert_{[\h,\h]}\ne 0$. Hence $[\h,\h]\ne 0$ and $\bar\xi:=\xi\vert_\h\ne 0$.
\end{proof}

For $\eta\in\co$, set $a_\eta=\codim (\q^\eta+\h)$ and $b_\eta=\dim (\q^\eta\cap\h)=\dim\h^\eta$. By 
Lemma~\ref{lm:1}{\sf (i)}, we have $a_\eta=\dim\co-\dim\h{\cdot}\eta$. Therefore, there is a 
dense open subset $\Omega_1\subset\Omega \subset\co$ such that the functions $\eta\mapsto a_\eta$ and
$\eta\mapsto b_\eta$ are constant on $\Omega_1$ and $\bar\eta\ne 0$ for all $\eta\in\Omega_1$.
Moreover, if $\eta\in\Omega_1$, then
\begin{gather*}
 a:=a_\eta=\min\{a_\xi\mid \xi\in\co\}=\min_{\xi\in\co} \codim (\q^\xi+\h)  \\ 
  b:=b_\eta=\min\{b_\xi\mid \xi\in\co\}=\min_{\xi\in\co} \dim (\q^\xi\cap\h) . 
\end{gather*}
Using this notation, we state the main result of this section.

\begin{thm}         \label{thm:main1}
Let $\h$ be a proper subalgebra of $\q$, with non-trivial $H$-action on $\co=Q{\cdot}\xi$. Then 
$\ind\h\le a+b$.
\end{thm}
\begin{proof}
Without loss of generality, we may assume that $\xi\in\Omega_1$. Hence $\bar\xi\ne 0$, $a=a_\xi$, and $b=b_\xi$. Since 
$p$ is $H$-equivariant, we have $p(H{\cdot}\xi)=H{\cdot}\bar\xi$. Note that $H{\cdot}\xi$ is just an 
$H$-orbit in $\q^*$, whereas $H{\cdot}\bar\xi$ is a coadjoint orbit in $\h^*$. It follows from 
Lemma~\ref{lm:1}{\sf (ii)} that 
\[
     \dim p(\co)=\dim H{\cdot}\xi=\dim\h-b .
\] 
In other words, $\codim_{\h^*} p(\co)=b$. By 
Lemma~\ref{lm:1}{\sf (i)}, we also have $\dim\co=\dim H{\cdot}\xi +a=\dim p(\co)+a$. Hence the dimension 
of generic fibres of $\hat p=p\vert_{\co}$  equals $a$. This implies that
$\dim H{\cdot}\xi-\dim H{\cdot}\bar\xi\le a$. Then
$\dim H{\cdot}\bar\xi\ge \dim H{\cdot}\xi-a= \dim\h-a-b$. Thus, $\ind \h\le a+b$.
\end{proof}

This immediately implies the following special cases of this theorem.
\begin{cl}   \label{cor:3-cases}
\leavevmode\par
\begin{itemize}
\item[\sf (i)] \ If\/ $\h\cap\q^\xi=\{0\}$ for some $\xi\in\co$, then $\hat p:\co\to\h^*$ is dominant, 
$\ind\h\le a$, and the generic fibres of $\hat p$ are of dimension $a$.
\item[\sf (ii)] \  If\/ $\h+\q^\xi=\q$ \ for some $\xi\in\co$, then $\dim H{\cdot}\xi=\dim\co$, $\ind\h\le b$, and
the generic fibres of $\hat p$ are finite.
\item[\sf (iii)] \  If\/ $a+b\le 1$, then $\ind\h=a+b$.
\end{itemize}
\end{cl}

\begin{rmk}         \label{rem:previous}
1. If $a+b=0$ (i.e., $\h\oplus\q^\xi=\q$), then $H{\cdot}\bar\xi$ is dense in $\h^*$ and $\ind\h=0$. This has 
been proved in \cite[Prop.\,5]{GS}, i.e., Theorem~\ref{thm:main1} is a generalisation of that result. This also 
generalises some observations of~\cite[6.4]{aif99}, where a similar method is applied to spherical nilpotent 
orbits in a reductive Lie algebra $\g$ and $\h=\be$, a Borel subalgebra of $\g$.

2. If $a+b=1$, then $\dim\h$ is odd and  $\ind\h=1$ for the parity reason. There are two possibilities 
for this. They correspond to either Corollary~\ref{cor:3-cases}{\sf (i)} with $a=1$, or Corollary~\ref{cor:3-cases}{\sf (ii)} with $b=1$.
\end{rmk}

As usual, we say that a Lie algebra $\h$ is {\it Frobenius}, if $\ind \h=0$.
\begin{ex}     \label{ex:g-simple}
If $\q$ is a simple Lie algebra, then $\lg Q{\cdot}\xi\rg=\q^*$ for all nonzero $\xi\in\q^*\simeq \q$. 
Therefore, Lemma~\ref{lm:2} and subsequent results apply to all proper subalgebras $\h\subset\q$ 
and all no-trivial $Q$-orbits.
\end{ex}

Theorem~\ref{thm:main1} is a tool for estimating $\ind\h$ via the use of coadjoint orbits of $Q$. The most 
interesting case occurs if $\h\oplus\q^\xi=\q$. Then $\h$ is Frobenius and we elaborate on this case in 
Section~\ref{sect:strange}.

By Vinberg's inequality~\cite[Prop.\,1.6 \& Cor.\,1.7]{p03},  
one has $\ind\q^\xi\ge\ind\q$ for any 
$\xi\in\q^*$. On the other hand, the celebrated "Elashvili conjecture" asserts that if $\q$ is reductive, then 
$\ind\q^\xi=\ind\q$  for all $\xi$, see~\cite{cm10} for a final chord in the proof and historical account. But it 
can happen that $\ind\q^\xi > \ind\q$ for some non-reductive $\q$ and 
$\xi\in\q^*$~\cite[Example\,1.1]{py06}. Here one may ask the following natural

\begin{qtn}   \label{qu:1}
Is it true that\/ $\ind \q^\xi=\ind\q$  whenever $\q^\xi\oplus\h=\q$?
\end{qtn}

\begin{rmk}
We see that if $\q^\xi\oplus\h=\q$, then $\ind\q^\xi+\ind\h\ge\ind\q$. More generally, the integer
$\ind\h+\ind\rr-\ind\q$ can be considered for any subalgebras $\h$ and $\rr$ such that $\h\oplus\rr=\q$.
Examples show that, for such an arbitrary triple $(\q,\h,\rr)$, there is no constraint on the sign of that 
difference. But, for some classes of triples, one obtains the ``expected'' answer. Let $\p$ be 
a parabolic subalgebra of a semisimple Lie algebra $\g$. If $\p^{\sf nil}$ is the nilradical of $\p$ and 
$\p^{\sf nil}_-$ is the nilradical of an
opposite parabolic subalgebra, then $\g=\p\oplus\p^{\sf nil}_-$ and
\[
    \ind\p+\ind\p^{\sf nil}_-=\ind\p+\ind\p^{\sf nil}\ge\rk\g=\ind\g .
\] 
The latter is conjectured in~\cite[Sect.\,6]{p03} and proved in~\cite{yu08}.
\end{rmk}

\section{On a question of Duflo}   \label{sect:quest}
\noindent
Let $G$ be a simple algebraic group with $\g=\Lie G$ and $H$ a proper connected subgroup of $G$. 
Using the Killing form, we identify $\g^*$ with $\g$ and deal with the adjoint representation of $\g$.
As above, we consider the projection $p:\g\to \h^*$. The following question was asked by 
Michel Duflo at the beginning of 2000's.

\begin{qtn}               
Suppose that $x\in\g_{\sf reg}$. Is it true that $p(G{\cdot}x)\cap\h^*_{\sf reg}\ne\varnothing$?
\end{qtn}
Using results of Section~\ref{sect:estimate-index} and Appendix~\ref{app}, we give an affirmative answer 
to this question. More generally,  we consider semisimple algebraic groups $G$ and connected subgroups 
$H$ that do not contain infinite normal subgroups of $G$. Such subgroups of $G$ are said to be 
{\it essentially proper}. If $G$ is simple, then "proper"="essentially proper".
Let us begin with a well-known simple assertion. 

\begin{lm}    \label{lm:ss}
Let $H$ be an essentially proper subgroup of a semisimple Lie group $G$. If $x\in\g_{\sf reg}$ is 
semisimple, then there is $y\in G{\cdot}x$ such that $\g^y\cap\h=\{0\}$. In particular, $\dim G/H\ge \rk G$.
\end{lm}
\begin{proof}
By the hypothesis, $T:=G^x$ is a maximal torus in $G$. Since $H$ is essentially proper, the action 
$(G:G/H)$ is {\it locally effective}, i.e., the kernel of action is finite. Then the same is true for  $(T:G/H)$. 
For the $T$-actions on irreducible algebraic varieties, generic stabilisers coincide with the kernel of 
actions~\cite[\S\,7.2]{vp}. Hence generic stabilisers for  the 
$T$-action on $G/H$ are finite. Thus,  $g{\cdot}\te\cap \h=\{0\}$ for some $g\in G$ and then 
$\g^y\cap\h=\{0\}$ for $y=g{\cdot}x$.
\end{proof}

However, the main point is to get the same result for the regular nilpotent elements. Then the passage to 
arbitrary regular elements of $\g$ is obtained via the use of invariant-theoretic properties of the 
adjoint representation of $G$. Let $\N=\N(\g)$ denote the nilpotent cone in $\g$. By~\cite{ko63}, $\N$
contains finitely many $G$-orbits. Furthermore, $\N_{\sf reg}:=\N\cap\g_{\sf reg}$ is the dense $G$-orbit in 
$\N$ and it is the smooth locus of $\N$. The elements of $\N_{\sf reg}$ are said to be 
{\it principal nilpotent}, and this $G$-orbit is sometimes denoted by $\co_{\sf pr}$.

\begin{prop}        \label{prop:nilp}
If $x\in\N_{\sf reg}$ and $H$ is an essentially proper subgroup of $G$, then there is 
$y\in G{\cdot}x$ such that $\g^y\cap\h=\{0\}$, i.e., the action $(H:\N)$ is locally free.
\end{prop}
\begin{proof}
Since $x$ is regular nilpotent, the group  $A:=(G^x)^0$ is unipotent and $G^x/(G^x)^0$ is the centre of 
$G$. It suffices to prove the assertion for the maximal subalgebras $\h$ of $\g$. It is known that a maximal 
subalgebra of $\g$ is either {\sl semisimple\/} or {\sl parabolic}~\cite[Theorem\,4]{K51}.

\textbullet \quad Let $\h=\p$ be a {\sl parabolic\/} subalgebra and $\p=\Lie P$. 
\\ Clearly, the action $(P:\N)$ is locally free if and only if $(A:G/P)$ is. To prove the latter, we take a simple 
$G$-module $\gV$ and a highest weight vector $v\in\gV$ such that $P$ is the stabiliser of the line $\lg v\rg\in\BP\eus V$. Then 
$\gC_{\eus V}:=\ov{G{\cdot}v}=G{\cdot}v\cup\{0\}$ is the affine cone over 
$G/P\simeq G{\cdot}\lg v\rg\subset \BP\eus V$. Since $A$ is unipotent, it suffices to prove that the action 
$(A:\gC_{\eus V})$ is locally free. This can be derived from a comparison of the actions of $T$ and $A$ on 
$\gC_{\eus V}$. By~\cite[Theorem\,1.5]{nilp-tori}, we have 
$\trdeg(\bbk[\gC_{\eus V}]^A)=\trdeg(\bbk[\gC_{\eus V}]^T)$. Furthermore, the $T$-action on
$\gC_{\eus V}$ is stable and locally free~\cite[Lemma\,2.1]{tg97}. Hence 
$\trdeg(\bbk[\gC_{\eus V}]^T)=\dim\gC_{\eus V}-\rk\g$. As $A$ has no rational characters, the field of 
$A$-invariant rational functions, $\bbk(\gC_{\eus V})^A$, is the field of fractions of the algebra 
$\bbk[\gC_{\eus V}]^A$. Then using the Rosenlicht theorem~\cite[\S\,2.3]{vp}, we obtain
\[
  \max_{z\in\gC_{\eus V}}\dim A{\cdot}z=\dim\gC_{\eus V}-\trdeg\bbk[\gC_{\eus V}]^T=\rk\g=\dim A .
\]
\indent \textbullet \quad Let $\h$ be a {\sl semisimple\/} subalgebra of $\g$. \\
In this case, we have a stronger result.
Recall that if $H$ is a reductive group and $X$ is an affine variety, then the action $(H:X)$ is said to be 
{\it stable}, if the union of closed $H$-orbits is dense in $X$. A purely algebraic approach to 
stability of actions, together with a number of astonishing new results, is given by 
Vinberg~\cite{vi00}.  Our stronger result is:
\\[.5ex]
\centerline{ \emph{If $H\subset G$ is reductive and essentially proper, then
the action $(H:\N)$ is stable and locally free.} } 
\vskip.4ex
(See Appendix~\ref{app} for the details and proof.) This completes the proof of proposition.
\end{proof}

\begin{prop}      \label{prop:any}
If $x\in\g_{\sf reg}$ and $H$ is an essentially proper subgroup of $G$, then there is $y\in G{\cdot}x$ such 
that $\g^y\cap\h=\{0\}$. 
\end{prop}
\begin{proof}
Let $\pi: \g\to\g\md G:=\spe(\bbk[\g]^G)\simeq \BA^{\rk\g}$ be the categorical quotient~\cite[\S\,4]{vp} for 
the adjoint action $(G:\g)$. Then $\N=\pi^{-1}(\pi(0))$ and each fibre of $\pi$ contains a dense orbit that 
belongs to $\g_{\sf reg}$~\cite{ko63}. Assuming that $x\not\in\N$, consider the line 
$\bbk x=\lg x\rg\subset\g$ and $\pi(\bbk x)\subset \g\md G$. Then $\pi(\bbk x)$ is a curve 
through $\pi(x)$ and $\pi(0)$. Since all fibres of $\pi$ are irreducible and have one and the same 
dimension $\dim\g-\rk\g$, the one-parameter family of fibres $X:=\pi^{-1}(\pi(\bbk x))\subset\g$ is 
irreducible, too. By Proposition~\ref{prop:nilp}, $(H:\N)$ is locally free. As $X\supset \N$, the action 
$(H:X)$ is also locally free. Hence the open subset $X\setminus \N$ contains points with a finite 
$H$-stabiliser. Since all fibres $\eus F_t=\pi^{-1}(\pi(tx))$, $t\ne 0$, are isomorphic $G$-varieties, each 
$\eus F_t$ contains points with finite $H$-stabilisers. Since $G{\cdot}x$ is the dense orbit in $\eus F_1$, 
the same is true for $G{\cdot}x$, and we are done.
\end{proof}

\begin{thm}   \label{thm:duflo}
Let $H$ be an essentially proper subgroup of $G$ and $p:\g\simeq\g^*\to\h^*$ the corresponding 
projection. If $x\in\g_{\sf reg}$, then $p(G{\cdot}x)$ is dense in $\h^*$. In particular, 
 $p(G{\cdot}x)\cap\h^*_{\sf reg}\ne\varnothing$.
\end{thm}
\begin{proof}
By Proposition~\ref{prop:any}, there is $y\in G{\cdot}x$ such that $\g^y\cap\h=\{0\}$. It then follows from
Corollary~\ref{cor:3-cases}{\sf (i)} that $p(G{\cdot}x)$ is dense in $\h^*$. 
\end{proof}

\section{Strange coadjoint orbits}
\label{sect:strange}
\noindent
When I was visiting Universit\'e de Poitiers at the beginning of 2000's, M.\,Ra\"\i s told me about  a certain 
class of orbits in semisimple Lie algebras, which he suggested to call ``strange'' (see below).  His interest 
in these orbits was motivated by a connection with the transverse Poisson structure of coadjoint orbits, 
see e.g.~\cite{cr,oh}. Soon afterwards, I explained to H.~Sabourin (Poitiers) how to prove that all 
spherical $G$-orbits in a reductive Lie algebra $\g$ are strange (via the use of my theory of doubled 
actions~\cite{disser}). Then he took on himself the task of publishing my proof in~\cite{s05}. 
Below, I present  some other results of mine related to strange orbits. 

Following Ra\"\i s, we give the following general definitions.

\begin{df} 
For an arbitrary Lie algebra $\q$, the coadjoint $Q$-orbit $\co\subset \q^*$ is said to be {\it strange}, if 
there is a subalgebra $\h\subset\q$ and $\xi\in\co$ such that $\q^\xi\oplus\h=\q$. Then we say that $\h$ 
is a {\it complementary subalgebra\/} for $\q^\xi$ (or for $\co$) and $(\co,\h)$ is a {\it strange pair}.
\end{df}

\noindent
For a strange pair,  $\dim\co=\dim\h$ and $\h$ is Frobenius by Corollary~\ref{cor:3-cases}{\sf (iii)}. Hence 
$\h$ is neither reductive nor nilpotent. In this case, $H$ has a dense orbit in $\co$ and the action
$(H:\co)$ is locally free.
It might be interesting to 
explore strange orbits for various classes of Lie algebras; especially, in connection with 
Question~\ref{qu:1}.

In the rest of this section, we assume that $Q=G$ is reductive, hence $\g^*\simeq\g$ as $G$-module. For 
any $m\in\BN$, consider the locally closed subset $\g^{(m)}=\{x\in\g\mid \dim G{\cdot}x=m\}$. The 
irreducible components of all these subsets are called {\it sheets}. The {\it rank\/} of a sheet $\gS$ is
$\rk\gS=\dim(\bar\gS\md G)$. For instance, if $m=\dim\g-\rk\g$, then
$\g^{(m)}=\g_{\sf reg}$ is the 
unique open sheet in $\g$ and $\rk\g_{\sf reg}$ is the usual rank of $\g$. Sheets have first been studied by Dixmier for $\g=\sln$~\cite{dixm}. The basic 
general results are obtained in~\cite{bo81,bk79}. In particular, each sheet contains a unique nilpotent 
orbit. However, if $\g\ne \sln$, then a nilpotent orbit can belong to several sheets, i.e., sheets are not 
necessarily disjoint. A sheet containing a semisimple orbit is said to be {\it Dixmier}. For $\sln$, all sheets are Dixmier (and disjoint).

\begin{thm}    \label{thm:sheets}
Suppose that a nilpotent $G$-orbit $\co$ is strange. Let $\gS_\co$ be a sheet of\/ $\g$ that contains 
$\co$. Then all $G$-orbits in $\gS_\co$ are strange. Moreover, if $(\co,\h)$ is a strange pair, then so is 
$(\tilde\co,\h)$ for any orbit $\tilde\co\subset\gS_\co$.
\end{thm}
\begin{proof}
For $x\in\co$, take an $\tri$-triple $\{x,h,y\}$~\cite[Chap.\,3]{CM}. Let $\lb:\bbk^*\to G$ be the 1-parameter subgroup such that
$\textsl{d}\lb(1)=h$. Then $\lb(t){\cdot}x=t^2x$ and $\lb(t){\cdot}y=t^{-2}y$. The eigenvalues of $h$ on 
$\g^y$ are non-positive. Therefore, the $1$-parameter group $t\mapsto\tilde\lb(t)=t^{-2}\lb(t)\subset GL(\g)$ has 
negative weights on $\g^y$ and also $\tilde\lb(t){\cdot}x=x$. Hence $\tilde\lb(t)$ acts on the affine space
$x+\g^y$ contracting everything to $x$, i.e., for any $z\in x+\g^y$, the closure of 
$\tilde\lb(\bbk^*){\cdot}z$ contains $x$.

According to Katsylo~\cite{kats82}, $\gK:=(x+\g^y)\cap\gS_\co$ is a conical  section of $\gS_\co$. Let 
$\tilde\co\subset\gS_\co\setminus \co$ be another $G$-orbit and $z\in \tilde\co\cap\gK$. Since 
$x\in\ov{\tilde\lb(t){\cdot}z}$ and the $G$-orbits $G{\cdot}(\tilde\lb(t){\cdot}z)=t^{-2}G{\cdot}z$ are 
isomorphic, one has $\tilde\lb(t){\cdot}z\in\gK$ for all $t\in\bbk^*$. If $\h$ is a complementary subalgebra 
for $\g^x$, then $\g^{\tilde\lb(t){\cdot}z}\cap\h=\{0\}$ for all but finitely many $t\in\bbk^*$. Since the 
orbits $G{\cdot}(\tilde\lb(t){\cdot}z)$, $t\ne 0$, are isomorphic, all of them are strange. This also shows that
$\h$ a complementary subalgebra for all orbits in $\gS_\co$.
\end{proof}

\begin{rmk}    \label{rem:1-PS-known}
The affine space $x+\g^y$ and 1-parameter group $\tilde\lb$ first appeared in the 
seminal paper of Kostant~\cite{ko63} for $x\in\N_{\sf reg}=\co_{\sf pr}$. Then such a transversal space was considered for arbitrary 
nilpotent orbits in $\g$ and even more general $G$-modules, see 
e.g. \cite{kats82},\,\cite[7.4]{slod},\,\cite[\S\,8.8]{vp}. 
Nowadays, $x+\g^y$ is usually called a {\it Slodowy slice} for the orbit $G{\cdot}x=G{\cdot}y$.
\end{rmk}

\begin{ex}    \label{ex:sheet-sl}
{\sf 1.} If $x$ is a principal nilpotent element of $\gln$, then $(\gln)^x=\lg I,x,\dots,x^{n-1}\rg$.  
Assume that $x=\sum_{j=1}^{n-1} e_{j,j+1}$, where $e_{i,j}$ is the usual matrix unit. Let $\h$ be the 
subalgebra of $\gln$ such that either the first row or the last column is zero. Then $(\gln)^x\oplus\h=\gln$ 
and thereby $GL_n{\cdot}x$ is strange. We generalise this observation in Section~\ref{sect:str-glv}.
Next, take $z=e_{n,1}+x$. Then $z$ is regular semisimple, i.e., 
$z$ belongs to the open sheet $(\gln)_{\sf reg}$, and the same $\h$ is also a complementary subalgebra for $(\gln)^z$.

{\sf 2.} If $x$ is considered as element of $\sln$, then the construction of $\h$ has to be modified. Here
$(\sln)^x=\lg x,\dots,x^{n-1} \rg$ and the complementary subalgebra is $\tilde\h=(\h\oplus \bbk I)\cap\sln$, which is a parabolic subalgebra of $\sln$ of maximal dimension. 
\end{ex}

\begin{qtn}    \label{qtn:3}
Let $\co$ be a strange non-nilpotent orbit in $\g$. Is it true that the nilpotent orbit of the sheet containing 
$\co$ is strange as well?
\end{qtn}

Let $B$ be a Borel subgroup of $G$. Following Vinberg, we say that the {\it complexity\/} of an irreducible 
$G$-variety $X$ is $c_G(X):=\dim X-\max_{x\in X}\dim B{\cdot}x$. By the Rosenlicht theorem, one has 
$c_G(X)=\trdeg \bbk(X)^B$. Then $X$ is said to be {\it spherical}, if $c_G(X)=0$, i.e., $B$ has a dense 
orbit in $X$. We only use this notion for $G$-orbits in $\g$. If $\g$ is simple, then a necessary 
condition of sphericity for $\co=G{\cdot}x$ is that the centraliser of $\g^x$ in $\g$  is one-dimensional.

\begin{thm}  \label{thm:c0}
If $\co$ a spherical $G$-orbit in $\g$, then $\co$ is strange. Moreover, there is a \emph{solvable} 
complementary subalgebra for $\co$. 
\end{thm}
\begin{proof}
Let us recall rudiments of the theory of doubled actions~\cite[Ch.\,1]{disser}. Fix a Borel subgroup 
$B\subset G$ and a maximal torus $T\subset B$. Let $\Pi$ denote the set of simple roots \wrt\  
$(B,T)$. We say that $t\in\te=\Lie T$ is dominant, if $\ap(t)\ge 0$ for all $\ap\in\Pi$. Let $\vartheta$ be
an involution of $G$ such that $\vartheta(s)=s^{-1}$ for all $s\in T$. For any $G$-orbit $\co$ and the 
action $(B:\co)$, there is a generic point    
$x\in\co$ such that the stabilisers $G^x$ and $B^x$ have a number of good properties:
\begin{itemize}
\item $S:=G^x\cap \vartheta(G^x)$ is reductive (possibly non-connected);
\item there is a dominant $t\in \te$ such that $(G^t,G^t)\subset S\subset G^t$;
\item $(B^x)^0=(B\cap S)^0$ is a Borel subgroup of $S^0$;
\item $(S\cap T)^0$ is a maximal torus in $S^0$.
\end{itemize}
Let $\Pi(\es)$ be the set of simple roots of $\es$ \wrt\ $(\be\cap\es, \te\cap\es)$. Then
$\Pi(\es)=\{\ap\in\Pi\mid \ap(t)=0\}$.
Therefore, there is a $T$-stable subalgebra $\tilde\be\subset \be$ such that $\be^x\oplus\tilde\be=\be$. Since $t$ is dominant, $\p:=\g^t+\be$ is a (parabolic) subalgebra of $\g$ and $\tilde\be\cap [\be,\be]$ is the nilradical of $\p$.

If $\co$ is spherical, then $B{\cdot}x$ is dense in $\co$. Therefore, $\g^x+\be=\g$ and hence 
$\g^x\oplus\tilde\be=\g$.
\end{proof}

\begin{rmk}    \label{rmk:sph-sheet}
For any sheet $\gS\subset\g$, all orbits in $\gS$ have the same complexity~\cite[Prop.\,4.5.23]{disser}.
In particular, if $\co\subset\N$ is spherical, then all orbits in $\gS_\co$ are spherical and thereby strange. 
The latter also follows from Theorem~\ref{thm:sheets}. 
However, if $\co$ is strange but not spherical, then a complementary subalgebra for 
$\co$ cannot be solvable. 
\end{rmk}

Recall that an orbit $\co\subset\N$ is said to be {\it rigid}, if $\co$ itself is a sheet. Then $\gS_\co=\co$ is the only 
sheet containing $\co$ and its rank is $0$. At the other extreme, $\co$ is said to be {\it Richardson} (or 
{\it polarisable}), if it is contained in a Dixmier sheet. (We refer to~\cite[Chap.\,7]{CM} for generalities on 
these classes of orbits.)
There exist also sheets of positive rank that are not 
Dixmier. In view of Theorems~\ref{thm:sheets} and~\ref{thm:c0}, it is of interest to know what spherical 
nilpotent orbits are not rigid. 

\begin{prop}      \label{prop:dichotomy}
Let $\g$ be a simple Lie algebra.
\begin{itemize}
\item[\sf (1)] \ If\/ $\co=G{\cdot}x\subset\g$ is spherical, then $x$ is either nilpotent or semisimple. 
\item[\sf (2)] \ If\/ $\co\subset\N$ is spherical, then it is either rigid or Richardson. 
\item[\sf (3)] \ if\/ $\gS_\co$ is a Dixmier sheet such that $\co$ is spherical, then 
$\rk\gS_\co=1$.
\end{itemize}
\end{prop}
\begin{proof}
(1) Let $x=x_s+x_n$ be the Jordan decomposition, i.e., $[x_s,x_n]=0$ and $\g^x=\g^{x_s}\cap\g^{x_n}$.
If $x$ is neither semisimple nor nilpotent, then the centraliser of $\g^x$ in $\g$ contains both $x_s$ and $x_n$.
Hence its dimension $\ge 2$. Therefore $G{\cdot}x$ is not spherical.

(2) It follows from (1) and Remark~\ref{rmk:sph-sheet} that if $\gS_\co\ne\co$, then 
$\gS_\co\setminus\co$ consists of semisimple elements. In particular, $\gS_\co$ is Dixmier.

(3) Let us prove that, for a Dixmier sheet $\gS_\co$ with $\rk \gS_\co>1$, the semisimple orbits in 
$\gS_\co$ are not spherical. If $\gS_\co$ is a Dixmier sheet, then there is a parabolic subalgebra 
$\p=\el\oplus\p^{\sf nil}$ with a Levi subalgebra $\el$ such that $\ov{\co}=G{\cdot}\p^{\sf nil}$ and 
$\ov{\gS_\co}=\ov{G{\cdot}\z(\el)}$, where $\z(\el)$ is the centre of $\el$. In this case
$\rk\gS_\co=\dim\z(\el)$ and the semisimple orbits in $\gS_\co$ are isomorphic to $G/L$. 
Then $\g^L=\z(\el)\ne 0$ and for a spherical semisimple orbit, one must have $\dim\g^L=1$.
\end{proof}
However, there are Dixmier sheets $\gS$ such that
$\rk\gS=1$, but the orbits in $\gS$ are not spherical.
Let us give more details on the dichotomy in Prop.~\ref{prop:dichotomy}(2). The spherical nilpotent orbits 
in the simple Lie algebras are classified in~\cite[Ch.\,4]{disser}. If $\co=G{\cdot}e\subset\N$,
then $\co$ is spherical if and only if $(\ad\,e)^4=0$~\cite[Theorem\,4.2.6]{disser}. For the exceptional Lie 
algebras, this leads to a short list of such orbits; while for the classical series, this can be stated in terms 
of the partition classification of nilpotent orbits, see~\cite[Chap.\,4.3]{disser}. Then we have
\\ \indent
\textbullet \quad
for $\GR{G}{2},\GR{F}{4},$ and $\GR{E}{8},$ the spherical nilpotent orbits are rigid;
\\  \indent
\textbullet \quad for $\g=\sln$, all nilpotent orbits are Richardson;
\\ \indent
\textbullet \quad in all other types, there are both rigid and Richardson spherical nilpotent orbits. 

\noindent
For instance, if $\g$ is of type $\GR{E}{6}$ (resp. $\GR{E}{7}$), then $2\mathsf A_1$ (resp. 
$3\mathsf A''_1$) is the only spherical Richardson orbit. We refer to \cite[Chap.\,8]{CM} for the standard 
notation on nilpotent orbits in the exceptional Lie algebras. Thus, for a spherical Richardson orbit, one 
obtains the whole Dixmier sheet that consists of strange orbits. 

\begin{rmk}
The sheets of the simple Lie algebras are classified by G.\,Kempken~\cite{kemp} (the classical Lie 
algebras) and A.G.\,Elashvili~\cite{ag85} (the exceptional Lie algebras). (Elashvili's results have earlier 
been advertised by Spaltenstein in~\cite[Chap.\,II, Appendice]{spalt}, and there is also a more recent 
computer verification in~\cite{deG-AG}.)  Using these classifications, one notices that if $\co$ is
spherical and Richardson, then it is contained in a unique (Dixmier) sheet. Together with 
Proposition~\ref{prop:dichotomy}{\sf (2)}, this means that any spherical nilpotent orbit is contained in a unique
sheet.
\end{rmk}
Although spherical nilpotent orbits occur in every simple Lie algebra, there is only one series of
nilpotent orbits $\co$ with $c_G(\co)=1$~\cite[Theorem\,4.5.19]{disser}. Namely, this happens  for $G=SL_n$ ($n\ge 3$) and the orbit $\co(\blb)$ with partition $\blb=(3,1,\dots,1)=(3,1^{n-3})$.

\begin{prop}     \label{prop:strange:c=1}
For any $n\ge 3$, the nilpotent orbit $\co=\co(3,1^{n-3})\subset\g=\sln$ is strange. 
\end{prop}
\begin{proof} If $n=3$, then $\co$ is the principal nilpotent orbit, and we refer to Example~\ref{ex:sheet-sl}.

Assume that $n\ge 4$. Let $x\in\co$ be a generic point satisfying properties in the proof of 
Theorem~\ref{thm:c0}. That is, $S:=G^x\cap \vartheta(G^x)$ is reductive, etc.
Since $c_G(\co)=1$, we have $\codim(\g^x+\be)=1$. This readily implies that there is 
$\ap\in\Pi$ such that $\g^x+\p_\ap=\g$, where $\p_\ap:=\be\oplus\g_{-\ap}$ is the minimal parabolic 
subgroup corresponding to $\ap\in\Pi$. Actually, one can take any $\ap\in\Pi\setminus\Pi(\es)$. 
Then $(\p_\ap)^x=\g^x\cap\p_\ap=\g^x\cap\be=\be^x$  has all good properties described in 
Theorem~\ref{thm:c0}. By the construction, $\g_{-\ap}\not\subset \g^x$. Hence 
$\g_{\ap}\not\subset \g^x$ as well. (Otherwise, we would obtain $\g_\ap\subset \be\cap\g^x\subset\es$.) 
Take $\tilde\be\subset\be$ as in Theorem~\ref{thm:c0}. Then $\tilde\h:=\tilde\be\oplus\g_{-\ap}$ is a 
$T$-stable complement to $\g^x$ in $\g$. It remains to show that $\tilde\h$ is a subalgebra. It is easily seen that it suffices to have the property that $\ap$ is orthogonal to $\Pi(\es)$.

Using the algorithm for computing $\es$ for nilpotent orbits~\cite[Ch.\,4]{disser}, one shows that 
$\es\simeq\mathfrak{gl}_{n-4}$ for $n\ge 4$. Hence $\Pi(\es)=\varnothing$ for $n=4,5$, whereas
$\Pi(\es)=\{\ap_3,\dots,\ap_{n-3}\}$ for $n\ge 6$
(with the usual numbering of the simple roots of $\sln$).
Therefore, one can always take $\ap$ to be either $\ap_1$ or $\ap_{n-1}$.
\end{proof}

It follows that the Dixmier sheet in $\sln$ containing $\co(3,1^{n-3})$ consists of strange orbits. Thus, combining
Theorems~\ref{thm:sheets} \& \ref{thm:c0}, and Proposition~\ref{prop:strange:c=1}, we obtain

\begin{thm}   \label{thm:c1}
If $G$ is simple and $\co\subset\g$ is a $G$-orbit with $c_G(\co)\le 1$, then $\co$ is strange.
\end{thm}

At the moment, non-spherical strange orbits are known only in $\sln$ (Example~\ref{ex:sheet-sl} and
Prop.~\ref{prop:strange:c=1}). In Sections~\ref{sect:str-glv} \& \ref{sect:two-col}, we shall discover 
some other non-spherical strange orbits in $\sln$. But in spite of all efforts, I was unable to detect
non-spherical strange orbits in the other simple Lie algebras.

\subsection{Some numerology}     \label{subs:numero}
Let $\g$ be  a simple Lie algebra. Set
\begin{itemize}
\item $\gM_{\sf sph}=\gM_{\sf sph}(\g):=\max\{\dim\co \mid \co \text{ is a spherical $G$-orbit in $\g$}\}$;
\item $\gM_{\sf str}=\gM_{\sf str}(\g)=\max\{\dim\co \mid \co \text{ is a strange $G$-orbit in $\g$}\}$;
\item $\gM_{\sf Fr}=\gM_{\sf Fr}(\g)=\max\{\dim\h \mid \h \text{ is a Frobenius subalgebra of }\g\}$.
\end{itemize}
By Corollary~\ref{cor:3-cases}{\sf (iii)} and Theorem~\ref{thm:c0}, we have 
$\gM_{\sf sph}\le \gM_{\sf str}\le \gM_{\sf Fr}$. It is also known that 
$\gM_{\sf sph}=\dim\be-\ind\be$~\cite[6.4]{aif99}. For $\sln$, one has $\gM_{\sf str}= \gM_{\sf Fr}=n^2-n$, 
which is the maximal dimension of arbitrary the $SL_n$-orbits in $\sln$ and arbitrary subalgebras of 
$\sln$.

\begin{prob}
 Determine $\gM_{\sf str}$ and $\gM_{\sf Fr}$ for all simple $\g$.
\end{prob}
One can verify directly that $\gM_{\sf str}= \gM_{\sf Fr}=\dim\be$ for $\GR{B}{2}$, $\GR{B}{3}$,
$\GR{C}{3}$, and $\GR{G}{2}$. We discuss this and some related conjectures in 
Section~\ref{sect:problems}.

\section{On strange orbits in $\g=\gln$}   
\label{sect:str-glv}

\noindent
If $e\in\sln$ is nilpotent, then $\blb(e)=\blb=(\lb_1,\lb_2,\dots)$ denote the corresponding partition of 
$n$, i.e., $\{\lb_i\}$ are the sizes of blocks in the Jordan normal form of $e$. Conversely, if $\blb$ is a partition of $n$, then $\co(\blb)$ is the corresponding nilpotent $SL_n$-orbit. Let 
$\widehat\blb=(\mu_1,\dots, \mu_m)$ denote the conjugate (dual) partition. Then 
$GL_n{\cdot}e=SL_n{\cdot}e$, 
$\dim SL_n{\cdot}e=n^2-\sum_{j=1}^m \mu_i^2$, and 
$\dim(\sln)^e=(\sum_{j=1}^m \mu_i^2)-1$, cf.~\cite[Chap.\,6.1]{CM}.
For the sake of completeness, we state the following simple observation.
\begin{lm}    \label{lm:sl-vs-gl}
If $x\in \sln$, then the orbit $GL_n{\cdot}x$ is strange if and only if $SL_n{\cdot}x$ is. If\/ $\h\subset\gln$ is a complementary subalgebra for $(\gln)^x$, then $(\h\oplus\bbk I)\cap\sln$ is complementary for $(\sln)^x$. Conversely, if\/ $\h\oplus (\sln)^x=\sln$, then $\h\oplus (\gln)^x=\gln$ as well.
\end{lm}
Because of this lemma, we can freely switch between strange orbits and complementary subalgebras for $\sln$ and $\gln$.

Let $\tilde e$ be a principal nilpotent element of $\gln$, hence $\blb(\tilde e)=(n)$. If $\tilde e^k\in\gln$ is the usual matrix power of $\tilde e$ and $n=kl+q$ with $0\le q<k$, then 
\[
\blb(\tilde e^k)=(\underbrace{l+1,\dots ,l+1}_{q},\underbrace{l,\dots, l}_{k-q})=:
((l+1)^q, l^{k-q}) .
\] 
Hence $\widehat\blb(\tilde e^k)=(k^l,q)$ and $\dim GL_n{\cdot}\tilde e^k=n(n-k)+q(k-q)$.

\begin{thm}         \label{thm:strange-gln}
For any $k\ge 1$, the nilpotent orbit $GL_n{\cdot}\tilde e^k$ is strange and a complementary
subalgebra $\h\subset\gln$ can be chosen as in Fig.~\ref{fig:h-e}.
\end{thm}
\begin{proof}
{\sf (1)}  Set $e=\tilde e^k$. Clearly, $\h$ is a subalgebra and $\dim\h=n(n-k)+q(k-q)=\dim GL_n{\cdot}e$, as required. Assume 
that the matrix $\tilde e$ has the usual Jordan form, with $n-1$ entries equal to $1$
(cf. Example~\ref{ex:sheet-sl}).
Let $I_m$ denote the identity 
matrix of order $m$. Then $e$ is the $n\times n$ matrix, where the upper-right (=\,north-east) block of size 
$n-k$ is $I_{n-k}$ and all other entries are zero.
Our goal is to prove that $\h\cap\g^{e}=0$. We argue by induction on $l=[n/k]$. 
\begin{center}
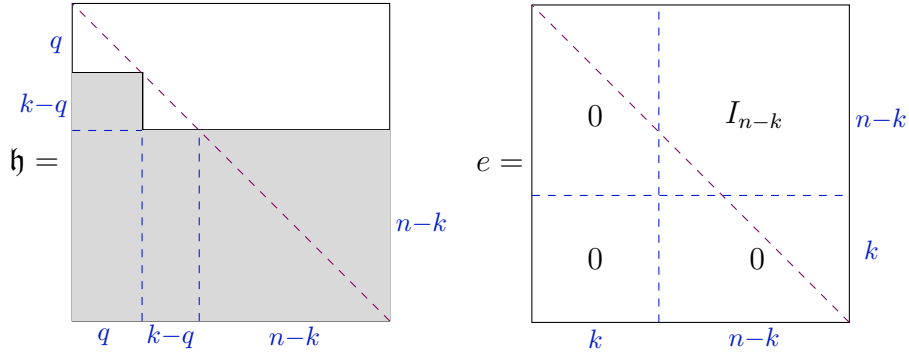
\begin{figure}[ht]           
\begin{tikzpicture}[scale= .42]
\draw (0,0)  rectangle (10,10);
\path[draw,line width=1pt]  (0,7.8) -- (2.2,7.8) -- (2.2,6) -- (10,6) ; 

\path[fill=gray!30]  (0,0) -- (10,0) -- (10,6) -- (0,6)--cycle ;
\path[fill=gray!30]  (0,6) -- (2.2,6) -- (2.2,7.8) -- (0,7.8)--cycle ;
\draw[dashed,redi]  (10,0) -- (0,10) ;
\draw[dashed,darkblue]  (0,6) -- (2.2,6) ;
\draw[dashed,darkblue]  (4,0) -- (4,6) ;
\draw[dashed,darkblue]  (2.2,0) -- (2.2,6) ;

\draw (1,-0.5)  node {\footnotesize  {\color{darkblue}$q$}} ;
\draw (3.1,-0.5)  node {\footnotesize {\color{darkblue}$k{-}q$}} ;
\draw (7,-0.5)  node {\footnotesize {\color{darkblue}$n{-}k$}} ;
\draw (-1.2,5)  node {$\h=$} ;
\draw (-0.9,6.8)  node {\footnotesize {\color{darkblue}$k{-}q$}} ;
\draw (-0.5,8.7)  node {\footnotesize {\color{darkblue}$q$}} ;
\draw (11,3.1)  node {\footnotesize {\color{darkblue}$n{-}k$}} ;
\end{tikzpicture}
\begin{tikzpicture}[scale= .42]
\draw (0,0)  rectangle (10,10);

\draw[dashed,redi]  (10,0) -- (0,10) ;
\draw[dashed,darkblue]  (0,4) -- (10,4) ;
\draw[dashed,darkblue]  (4,0) -- (4,10) ;

\draw (7,-0.5)  node {\footnotesize {\color{darkblue}$n{-}k$}} ;
\draw (2,-0.5)  node {\footnotesize {\color{darkblue}$k$}} ;
\draw (-1,5)  node {$e=$} ;
\draw (7,6.5)  node {$I_{n-k}$} ;
\draw (2,6.5)  node {$0$} ;
\draw (7.1,2)  node {$0$} ;
\draw (2,2)  node {$0$} ;
\draw (11,6.4)  node {\footnotesize {\color{darkblue}$n{-}k$}} ;
\draw (10.7,2.3)  node {\footnotesize {\color{darkblue}$k$}} ;
\end{tikzpicture}
\caption{A complementary subalgebra $\h$ for $(\gln)^e$ and $e=\tilde e^k$}  \label{fig:h-e}
\end{figure}
\end{center}
\vskip-2ex  
\par\indent
{\sf (2)} \ 
Suppose that $l=1$, i.e., $k>n/2$ and $n=k+q$. Then $\blb(e)=(2^q, 1^{k-q})$ 
and $\rk(e)=q$. In this case $e^2=0$, i.e., the orbit $GL_n{\cdot}e$ is spherical, see~\cite[Ch.\,4.3]{disser}.  Then 
\beq         \label{eq:dopolnit}
\text{
$\g^e=\left\{\begin{pmatrix}   A & \ast & \ast  \\ 0 & C & \ast \\  0 & 0 & A 
\end{pmatrix}\in \mathfrak{gl}_{k+q} \mid A\in \mathfrak{gl}_q, \ C\in \mathfrak{gl}_{k-q}
\right\}$ \ and \     
$\h=\begin{pmatrix}   0 & 0 & 0  \\ \ast & 0 & 0 \\  \ast & \ast & \ast
\end{pmatrix}$,} 
\eeq
cf. Fig.~\ref{fig:h-e}. Clearly, we have here $\g^e\oplus\h=\g$.

{\sf (3) The general case with} $l>1$. Our induction step provides the passage from $n=kl+q$ to
$n-k=k(l-1)+q$.  Recall that $\blb(e)=((l+1)^q, l^{k-q})$ and therefore $\dim\Ker(e)=k$.

We have $\gln=\glv$, where $\VV\simeq\bbk^n$. Consider the decompositions
\beq    \label{eq:decomp}
     \VV=\Ker(e)\oplus \UU=\Ima(e)\oplus \WW ,
\eeq
where $\dim\Ker(e)=\dim\WW=k$ and $\dim\Ima(e)=\dim\UU=n-k$. Since 
\[ 
\glv=\mathfrak{gl}\bigl(\Ima(e)\oplus \WW\bigr)= 
\mathfrak{gl}(\Ima(e))\oplus \Hom\bigl(\Ima(e),\WW\bigr)\oplus \Hom\bigl(\WW, \Ima(e)\bigr)\oplus \mathfrak{gl}(\WW),
\]  
we obtain the corresponding matrix decomposition of $\glv$:

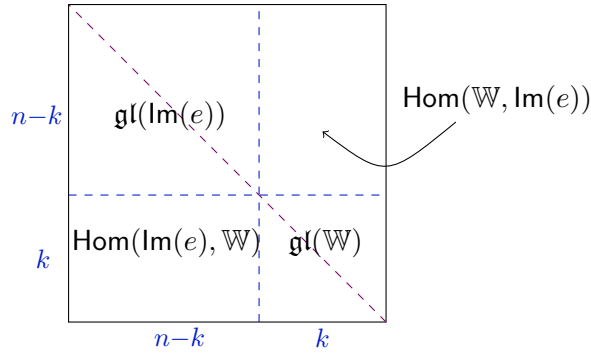
\begin{center}
\begin{tikzpicture}[scale=.42]
\draw (0,0)  rectangle (10,10);
\draw[dashed,redi]  (10,0) -- (0,10) ;
\draw[dashed,darkblue]  (6,0) -- (6,10) ;
\draw[dashed,darkblue]  (0,4) -- (10,4) ;

\draw (-1,6.5)  node {\footnotesize {\color{darkblue}$n{-}k$}} ;
\draw (-0.8,2)  node {\footnotesize {\color{darkblue}$k$}} ;
\draw (3.5,-0.5)  node {\footnotesize {\color{darkblue}$n{-}k$}} ;
\draw (8,-0.5)  node {\footnotesize {\color{darkblue}$k$}} ;
\draw (3.2,6.5)  node {\small $\mathfrak{gl}(\Ima(e))$} ;
\draw (3.1,2.5)  node {\small $\Hom(\Ima(e),\WW)$} ;
\draw (8.1,2.5)  node {\small $\mathfrak{gl}(\WW)$} ;
\draw (13.5,7)  node {\small $\Hom(\WW,\Ima(e))$} ;
\draw[->]   (12.2, 6.3) .. controls (10,4.5) .. (8,6);

\draw (-3.3,6.8)  node {\phantom{$\Hom(\WW,\Ima(e))$}} ;
\end{tikzpicture}
\end{center}
\noindent
If $l\ge 2$, then $\Ker(e)\subset\Ima(e)$. Therefore, one can assume that $\WW\subset \UU$ and 
arrange the refined decomposition
\beq \label{eq:refined}
     \VV=\Ker(e)\oplus \VV'\oplus \WW ,
\eeq
where $\dim\VV'=n-2k$, $\Ker(e)\oplus \VV'=\Ima(e)$, and $\VV'\oplus \WW=\UU$. The subspace 
$\Ima(e)$ is $e$-stable and $e$ induces a nilpotent transformation of $\Ima(e)$, which is denoted by $e'$. 
Set $\g'=\mathfrak{gl}(\Ima(e))\simeq \mathfrak{gl}_{n-k}$. It is easily seen that $\blb(e')=(l^q, (l-1)^{k-q})$, 
which means that $e'$ is the $k$-th associative power of a regular nilpotent element of $\g'$.
Since $\widehat\blb(e')=(k^{l-1},q)$, we have 
 \beq    \label{eq:raznost}
       \dim\g^e-\dim (\g')^{e'}=k^2 .
\eeq
By the induction assumption for $(\g')^{e'}$, there is the complementary subalgebra 
$\h'\subset\g'$ of prescribed shape. 
                     
By \eqref{eq:refined}, we have 
$\Hom(\WW,\VV') \oplus \gl(\WW)=\Hom(\WW,\UU)$.
Combining this with the presentation of $\h$ in Fig.~\ref{fig:h-e}, we see that
$\h=\h'\oplus\Hom\bigl(\Ima(e),\WW\bigr) \oplus \Hom(\WW, \UU)$, which can graphically be depicted 
as the coloured area below: 

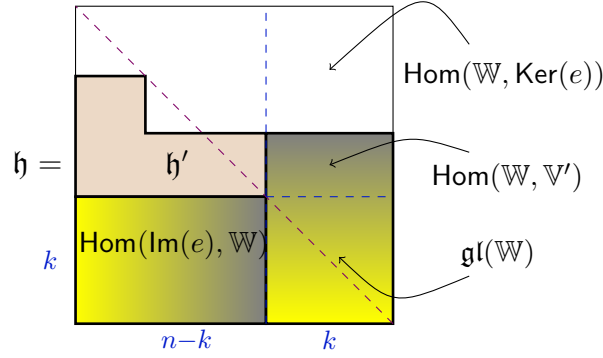
\begin{center}
\begin{tikzpicture}[scale=.42]
\draw (0,0)  rectangle (10,10);
\path[draw]  (2.2,6) -- (10,6); 
\path[draw]  (0,7.8) -- (2.2,7.8); 
\path[draw]  (2.2,6) -- (2.2,7.8);

\path[fill=brown!30,draw,line width=1pt] (0,4) -- (6,4) -- (6,6) -- (2.2,6) -- (2.2,7.8) -- (0,7.8) -- (0,4)--cycle ;
\shade[left color=yellow,right color=gray, draw,line width=1pt] (0,0) -- (6,0) -- (6,4) -- (0,4)--cycle ;
\shade[bottom color=yellow,top color=gray,draw,line width=1pt] (6,0) -- (10,0) -- (10,6) -- (6,6)--cycle ;
\draw[dashed,redi]  (10,0) -- (0,10) ;
\draw[dashed,darkblue]  (6,0) -- (6,10) ;
\draw[dashed,darkblue]  (6,4) -- (10,4) ;

\draw (3.1,2.5)  node {\small $\Hom(\Ima(e),\WW)$} ;
\draw (13.3,2.2)  node {\small $\gl(\WW)$} ;
\draw[->]   (12.4, 1.5) .. controls (10,.5) .. (8.3,2.1);

\draw (13.6,4.6)  node {\small $\Hom(\WW,\VV')$} ;
\draw[->]   (12.4, 5.4) .. controls (10.5,6) .. (8,5);

\draw (13.5,7.8)  node {\small $\Hom(\WW,\Ker(e))$} ;
\draw[->]   (12.4, 8.8) .. controls (11,10.7) .. (8,8);

\draw (3.5,-0.5)  node {\footnotesize {\color{darkblue}$n{-}k$}} ;
\draw (8,-0.5)  node {\footnotesize {\color{darkblue}$k$}} ;
\draw (-1.2,5)  node {$\h=$} ;
\draw (3.2,5)  node {$\h'$} ;
\draw (-0.8,2)  node {\footnotesize {\color{darkblue}$k$}} ;
\end{tikzpicture}
\end{center}
Consider the natural map $\Phi: \g^e\to (\g')^{e'}$ such that $x\mapsto x\vert_{\Ima(e)}=\Phi(x)$.
Since $xe=ex$, the subspace $\Ima(e)$ is $x$-stable and $\Phi$ is well-defined. Then 
\beq   \label{eq:Ker-Phi}
\Ker(\Phi)=\{x\in\g^e\mid x\vert_{\Ima(e)}=0\}=\{x\in\g^e\mid xe=0\}=\{x\in\g^e\mid ex=0\} .
\eeq
By~\eqref{eq:decomp}, the restriction $x\mapsto x\vert_{\WW}$ yields an injective linear map
$\Ker(\Phi)\to \Hom (\WW,\Ker(e))$. Since $\dim\Hom (\WW,\Ker(e))=k^2$, it follows 
from~\eqref{eq:raznost} that $\Ker(\Phi)\simeq \Hom (\WW,\Ker(e))$ and $\Phi$ is onto. 
Now, since 
$\h'\cap(\g')^{e'}=\{0\}$, using this description of $\Ker(\Phi)$ and $\h$, we conclude that
$\h\cap\g^e=\{0\}$, as required.
\end{proof}

\begin{ex}
If $k$ divides $n$, then $q=0$, $\dim(\gln)^e=nk=lk^2$, and $\h$ given in Fig.~\ref{fig:h-e}
is the set of matrices whose first $k$ rows are zero. This can also be explained as follows.
\\ \indent
If $n=kl$ and $e=\tilde e^k$, then $(\gln)^e$ can be obtained as a ``thick centraliser'' of a regular nilpotent element
of $\gl_l$. That is, if $\hat e\in (\gl_l)_{\sf reg}\cap\N$ has the usual Jordan form, then $(\gl_l)^{\hat e}$
is the set of upper-triangular matrices such that  $a_{11}=\dots=a_{ll}$, $a_{12}=\ldots=a_{l-1,l}$, etc.
Then we replace each $a_{ij}$ with $k\times k$-matrix $A_{ij}$ subject to the same relations. This 
provides $(\gln)^e$. For instance, if $n=3k$ (i.e., $l=3$), then 
\[
   e=  \begin{pmatrix}
0&I_k&0 \\ 0&0&I_k \\ 0&0&0
\end{pmatrix}\!, 
(\gln)^e=
\{\begin{pmatrix}
A_{11} & A_{12} & A_{13} \\ 0 & A_{11} & A_{12} \\ 0&0& A_{11}
\end{pmatrix}\mid A_{ij}\in \gl_k\}.
\ \text{ Hence } \ 
\h=\{ \begin{pmatrix}
 0 & 0 & 0 \\ \ast & \ast & \ast \\ \ast & \ast & \ast 
\end{pmatrix} \}
\]
 is a complementary subalgebra.
\end{ex}

\begin{rmk}
If $G{\cdot}e$ is strange, various complementary subalgebras for $\g^e$ may have essentially different structure. 
For instance, in Part (2) of the proof of Theorem~\ref{thm:strange-gln}, the Levi subalgebra of $\h$ 
is isomorphic to $\mathfrak{gl}_q$, see Eq.~\eqref{eq:dopolnit}. On the other hand, two diagonal 
blocks of size $q=n-k$ in $\h$ (the first of them is zero and the other is just $\mathfrak{gl}_q$) can be 
replaced with the first block
$D_1={\small\begin{pmatrix} t_1 & 0 & 0 \\  \ast & \ddots & 0\\ \ast & \ast & t_q  \end{pmatrix}}\in \mathfrak{gl}_q$  
and the second block
$D_2={\small\begin{pmatrix} -t_1 & \ast & \ast \\  0 & \ddots & \ast \\ 0 & 0 & -t_q \end{pmatrix}}
          \in \mathfrak{gl}_q$.  
This yields a solvable complementary subalgebra 
$\tilde\h={\small\begin{pmatrix}   D_1 & 0 & 0  \\ \ast & 0 & 0 \\  \ast & \ast & D_2
\end{pmatrix}}$.  
This reflects the general fact that, for a spherical orbit 
$G{\cdot}e$, there is always a solvable complementary subalgebra for $\g^e$.
\end{rmk}

\begin{rmk}       \label{rem:power-sph}
By~\cite[Ch.\,4]{disser}, a non-trivial orbit $\co(\blb)\subset\sln$ is spherical if and only if $\lb_1=2$. 
Therefore, all spherical nilpotent orbits are of the form $SL_n{\cdot}\tilde e^k$ for $k\ge \lfloor n/2\rfloor$.
But the construction of Theorem~\ref{thm:strange-gln} provide a non-solvable complementary subalgebra for them.
\end{rmk}

\section{Partitions with at most three parts and strange orbits}
\label{sect:two-col}
\noindent
We continue to work with nilpotent $SL_n$-orbits.
Let $\mathcal P(n)$ be the set of partitions of $n$. The standard partial order on $\mathcal P(n)$ is 
defined by condition that $\blb\succ\boldsymbol{\nu}$ if and only if $\sum_{i\le j}\lb_i\ge \sum_{i\le j}\nu_i$ 
for all $j$. Then $\ov{\co(\blb)}\supset \co(\boldsymbol{\nu})$ if and only if 
$\blb\succ\boldsymbol{\nu}$~\cite[Chap.\,6.2]{CM}.

\subsection{Partitions with two parts} A partition $\blb$ is said to be {\it two-column\/} if the number of its nonzero parts  is at most two. 
The largest two-column partition $(n)$ provides the principal nilpotent orbit 
$\co_{\sf pr}=SL_n{\cdot}\tilde e$ and the smallest two-column partition $(\lceil n/2\rceil, \lfloor n/2\rfloor)$ 
corresponds to $\co_{\sf pr}^{\lg 2\rg}:=SL_n{\cdot}\tilde e^2$. Hence
\[
   \co_{\sf pr}^{\lg 2\rg}=\begin{cases} \co(m,m), &  \text{if }\ n=2m, \\ 
   \co(m{+}1,m), & \text{if }\ n=2m{+}1.\end{cases}
\]
By Theorem~\ref{thm:strange-gln}, these orbits are strange, and we obtain below the complete description of strange 
$SL_n$-orbits between them, i.e., for all two-column partitions of $n$. It is assumed below that $n\ge 4$.

Let us begin with some preparations. Consider the usual numbering of the simple roots of $\sln$. If 
$A\subset\{1,2,\dots,n{-}1\}=:[n{-}1]$, then $\bar A=[n-1]\setminus A$ and
$\p_A$ is the standard parabolic subalgebra such that
$\{\ap_i\mid i\in A\}$ is the set of simple roots of the Levi subalgebra $\el_A\subset\p_A$. For instance, 
$\p_\varnothing=\be$ and, for any $ i\in[n{-}1]$, $\p_{\ov{i}}$ is a maximal 
parabolic subalgebra. (We prefer to write $\ov{i}$ in place of $\ov{\{i \}}$.) Then $\dim\p_{\ov{ i}}=
n^2{-}1-i(n{-}i)$ and $\p_{\ov{1}}$ is a subalgebra of maximal dimension.  By Example~\ref{ex:sheet-sl}(2),
$(\co_{\sf pr}, \p_{\ov{1}})$ is a strange pair, hence $\ind\p_{\ov{1}}=0$. Here
$\p_{\ov{1}}=\el_{\ov{1}}\ltimes \p_{\ov{1}}^{\sf nil}=\gl_{n-1}\ltimes \bbk^{n-1}$. The next possible 
dimension of parabolic subalgebras is $\dim\p_{\ov{2}}=\dim\p_{\ov{1}}-n+3=(n{-}1)^2{+}2$.
We also need $\p_{\ov{1,2}}=\p_{\ov{1}}\cap\p_{\ov{2}}$, where  $\dim\p_{\ov{1,2}}=\dim\p_{\ov{2}}-1$.

\begin{lm}       \label{lm:proper-p1}
Let\/ $\h$ be a proper subalgebra of $\p_{\ov{1}}$ such that 
$\dim\h\ge \dim\p_{\ov{1}}-(n-3)=\dim\p_{\ov{2}}$. Then $\h=\sel_{n-1}\ltimes\p_{\ov{1}}^{\sf nil}=[\p_{\ov{1}},\p_{\ov{1}}]$. (In particular, $\h$ is not Frobenius.)
\end{lm}
\begin{proof}
By the assumption, $\dim(\h\cap\sel_{n-1})\ge \dim\sel_{n-1}-(n-3)$. Since 
$\rk( \sel_{n-1})=n-2$, any proper subalgebra of $\sel_{n-1}$ is codimension $\ge n{-}2$, see~Lemma~\ref{lm:ss}. Therefore, $\h\supset \sel_{n-1}$.
\\  \indent 
Next, $\dim\h\cap\p_{\ov{1}}^{\sf nil}\ge (n{-}1)-(n{-}3)$ and $\p_{\ov{1}}^{\sf nil}$ is a simple 
$\sel_{n-1}$-module. Hence $\h\supset\p_{\ov{1}}^{\sf nil}$.
\end{proof}

Note that  $\p_{\ov{2}}$ is  not a subalgebra of $\p_{\ov{1}}$.
Let $\tilde\co\subset\N$ be the minimal orbit containing $\co_{\sf pr}^{\lg 2\rg}$ in its closure, i.e.,
$\tilde\co=\co(m{+}1,m{-}1)$ if $n=2m$; and $\tilde\co=\co(m{+}2,m{-}1)$ if $n=2m+1$.
\begin{thm}       \label{thm:no-strange}
If $\co\ne\co_{\sf pr}$ and $\dim\co>\dim\tilde\co$, then $\co$ is not strange.
\end{thm}
\begin{proof}
First, consider the case in which $n=2m+1$. Then $\dim\tilde\co=\dim\p_{\ov{2}}$. 
Assume that $\dim\tilde\co<\dim\co<\dim\co_{\sf pr}$ and $(\co,\q)$ is a strange pair. 
Then $\dim\p_{\ov{2}}<\dim\q<\dim\p_{\ov{1}}$ and $\ind\q=0$. Then $\q$ cannot be a parabolic 
subalgebra (for the dimension reason) and it is not semisimple. Hence $\q$ is not a maximal subalgebra.
Let $\widehat\q$ be a maximal subalgebra of $\sln$ containing $\q$. Then $\widehat\q$ is not semisimple, 
since there are no semisimple subalgebras of dimension bigger than $(n-1)^2$. Hence the only possibility is $\widehat\q=\p_{\ov{1}}$. Then Lemma~\ref{lm:proper-p1} leads to a contradiction.

For $n=2m$, the argument is similar. The difference is that now  $\dim\tilde\co=\dim\p_{\ov{1,2}}$, and the 
subalgebra $\p_{\ov{1,2}}$ is used in place of $\p_{\ov{2}}$. We also need the fact that $\p_{\ov{2}}$
is not Frobenius here (because $\dim\p_{\ov{2}}$ is odd for even $n$).
\end{proof}
\begin{rmk}
This theorem applies not only to the orbits corresponding to two-column partitions. It can happen that 
$\dim\co>\dim\tilde\co$, but $\ov{\co}\not\supset \tilde\co$, i.e., $\blb(\co)$ has more than two parts. The 
smallest example occurs for $n=10$, where $\tilde\co=\co(6,4)$ and $\dim\co(8,1,1)>\dim\co(6,4)$. Therefore, $\co(8,1,1)$ is not 
strange.
\end{rmk}

It follows from Theorem~\ref{thm:no-strange} that, for the two-column partitions, the only unclear case  
concerns the orbit $\tilde\co$. Our next goal is to handle this.

\begin{thm}    \label{thm:2-col-strange}
For any $n\ge 4$, the orbit $\tilde\co\subset\sln=\sel(\gV)$ is strange and there is a complementary 
parabolic subalgebra. More precisely, 
\\ \indent
{\sf (i)} \ If\/ $n=2m$ and $\tilde\co=\co(m{+}1,m{-}1)$, then a complementary subalgebra is\/ 
$\p_{\ov{1,2}}$.
\\ \indent
{\sf (ii)} \ If\/ $n=2m{+}1$ and $\tilde\co=\co(m{+}2,m{-}1)$, then a complementary subalgebra is\/ 
$\p_{\ov{2}}$;
\end{thm}
\begin{proof} Let $\tilde\p$ denote $\p_{\ov{2}}$ (resp. $\p_{\ov{1,2}}$) if $n$ is odd (resp. even). In both 
cases, $\dim\tilde\p=\dim\tilde\co$, hence $\dim\tilde\p$ is even. It then suffices to prove that, for 
$x\in\tilde\co$, the group $(SL_n)^x$ has an open orbit in the flag variety $SL_n/\tilde P$. Since 
$(\dim (SL_n)^x=\dim SL_n/\tilde P$, it is also sufficient to find a flag $\eus F\in SL_n/\tilde P$ whose
stabiliser in $(SL_n)^x$ is finite. Thus, for a given $x\in\tilde\co$, we have to point out a flag $\eus F$ such that $(\sln)^x\cap {\sf stab}(\eus F)=\{0\}$. In both cases, we take $x$ in the Jordan normal form.
Then $(\gln)^x$ has a nice graphical description pointed out in~\cite[\S\,4.2]{ar71}. 
For the reader's convenience, we provide a sample picture for $x$ having two Jordan blocks.
\begin{figure}[htb]  
\begin{center}  
\begin{tikzpicture}[scale= .35] 
\draw (0,0)  rectangle (10,10);
\draw[thick,brown]  (0,4) -- (10,4) ;
\draw[thick,brown]  (6,0) -- (6,10) ;
\path[draw, line width=1pt]  (.2,9.8)--(5.8,4.2);   
\path[draw, line width=1pt]  (1.2,9.8)--(5.8,5.2);
\path[draw, line width=1pt]  (2.2,9.8)--(5.8,6.2);
\path[draw, line width=1pt]  (3.2,9.8)--(5.8,7.2);
\path[draw, line width=1pt]  (4.2,9.8)--(5.8,8.2);
\path[draw, line width=1pt]  (5.2,9.8)--(5.8,9.2);
\path[draw, line width=1pt]  (2.2,3.8)--(5.8,.2); 
\path[draw, line width=1pt]  (3.2,3.8)--(5.8,1.2);
\path[draw, line width=1pt]  (4.2,3.8)--(5.8,2.2);
\path[draw, line width=1pt]  (5.2,3.8)--(5.8,3.2);
\path[draw, line width=1pt]  (6.2,9.8)--(9.8,6.2);  
\path[draw, line width=1pt]  (7.2,9.8)--(9.8,7.2);
\path[draw, line width=1pt]  (8.2,9.8)--(9.8,8.2);
\path[draw, line width=1pt]  (9.2,9.8)--(9.8,9.2);
\path[draw, line width=1pt]  (6.2,3.8)--(9.8,.2);  
\path[draw, line width=1pt]  (7.2,3.8)--(9.8,1.2);
\path[draw, line width=1pt]  (8.2,3.8)--(9.8,2.2);
\path[draw, line width=1pt]  (9.2,3.8)--(9.8,3.2);
\foreach \x in {1,3,5,7,9,11}  \shade[ball color=blue] (\x/2,9.5) circle (1.6mm);  
\foreach \x in {3,5,7,9,11}  \shade[ball color=blue] (\x/2,8.5) circle (1.6mm);
\foreach \x in {5,7,9,11}  \shade[ball color=blue] (\x/2,7.5) circle (1.6mm);
\foreach \x in {7,9,11}  \shade[ball color=blue] (\x/2,6.5) circle (1.6mm);
\foreach \x in {9,11}  \shade[ball color=blue] (\x/2,5.5) circle (1.6mm);
\foreach \x in {11}  \shade[ball color=blue] (\x/2,4.5) circle (1.6mm);
\foreach \x in {5,7,9,11}  \shade[ball color=blue] (\x/2,3.5) circle (1.6mm);    
\foreach \x in {7,9,11}  \shade[ball color=blue] (\x/2,2.5) circle (1.6mm);
\foreach \x in {9,11}  \shade[ball color=blue] (\x/2,1.5) circle (1.6mm);
\foreach \x in {11}  \shade[ball color=blue] (\x/2,.5) circle (1.6mm);
\foreach \x in {5,7,9,11}  \shade[ball color=blue] (4+\x/2,9.5) circle (1.6mm);    
\foreach \x in {7,9,11}  \shade[ball color=blue] (4+\x/2,8.5) circle (1.6mm);
\foreach \x in {9,11}  \shade[ball color=blue] (4+\x/2,7.5) circle (1.6mm);
\foreach \x in {11}  \shade[ball color=blue] (4+\x/2,6.5) circle (1.6mm);
\foreach \x in {5,7,9,11}  \shade[ball color=blue] (4+\x/2,3.5) circle (1.6mm);    
\foreach \x in {7,9,11}  \shade[ball color=blue] (4+\x/2,2.5) circle (1.6mm);
\foreach \x in {9,11}  \shade[ball color=blue] (4+\x/2,1.5) circle (1.6mm);
\foreach \x in {11}  \shade[ball color=blue] (4+\x/2,.5) circle (1.6mm);
\end{tikzpicture}
\end{center}
\caption{Centralisers of Jordan matrices in $\gln$}  \label{fig:centraliser}  \label{fig:arnold}
\end{figure}

Here the matrix entries along the oblique segments are equal, and all other entries are zero. Hence
$\dim(\gln)^x$ equals the number of segments. Of course, for $(\sln)^x$, one has to add the condition that
the trace equals zero.

{\sf (i)} \ For $\blb=(m{+}1,m{-}1)$ and $\gV=\bbk^{2m}$, we take a basis $(v_1,\dots,v_{2m})$ 
such that the Jordan blocks of $x$ are $\lg v_1,\dots,v_{m+1}\rg$ and $\lg v_{m+2},\dots,v_{2m}\rg$,
$\Ker(x)=\lg v_1,v_{m+2}\rg$, and $x{\cdot}v_j=v_{j-1}$ inside of each block.
Then $\tilde\p=\p_{\ov{1,2}}$ preserves a flag $\eus F=\{\gV_1\subset \gV_2\}$ and we take 
$\gV_1=\lg v_{m+1}\rg$, $\gV_2=\lg v_{m+1}, v_m+v_{2m}\rg$.
(Here $\gV_j$ is a subspace of $\gV$ of dimension $j$.) 

{\sf (ii)} \ For $\blb=(m{+}2,m{-}1)$ and $\gV=\bbk^{2m+1}$, we take a basis $(v_1,\dots,v_{2m+1})$ 
such that the Jordan blocks in $\gV$ are $\lg v_1,\dots,v_{m+2}\rg$ and $\lg v_{m+3},\dots,v_{2m+1}\rg$,
$\Ker(x)=\lg v_1,v_{m+3}\rg$, and $x{\cdot}v_j=v_{j-1}$ inside of each block. Here $\tilde\p=\p_{\ov{2}}$ 
preserves a flag $\eus F=\{\gV_2\}$ and we take the plane $\gV_2\subset\gV$ with 
basis $(v_{m+2}, v_{m+1}+v_{2m+1})$.

Using the description of the centraliser of Jordan matrices (Fig.~\ref{fig:centraliser}), one readily computes that if $z\in(\sln)^x$ preserves the chosen flag $\eus F$, then $z=0$ (in both cases).
\end{proof}

\begin{cl}    \label{cor:sledstvie}
\begin{itemize}
\item[{\sf (i)}] $\tilde\co$ is the only strange orbit strictly between $\co_{\sf pr}$ and $\co_{\sf pr}^{\lg 2\rg}$;
\item[{\sf (ii)}] $\ind\p_{\ov{1,2}}=0$ \ $\Leftrightarrow \ \dim\p_{\ov{1,2}}$ is even \ $\Leftrightarrow$ \ $n$ is even 
\item[{\sf (iii)}] $\ind\p_{\ov{2}}=0$ \ $\Leftrightarrow \ \dim\p_{\ov{2}}$ is even \ $\Leftrightarrow$ \ $n$ is odd;
\end{itemize}
\end{cl}
Of course, if it is proved somehow that $\p$ is a complementary parabolic subalgebra, then $\ind\p=0$. 
But in order to pick a right candidate up, it is better to know beforehand that $\p$  is Frobenius (and 
$\dim\p=\dim\co$). The parabolic subalgebras of semisimple Lie algebras form a special class of 
{\it seaweed subalgebras\/} (=\,{\it seaweeds}), and there is a general formula for the index of seaweeds.
That is, regardless of strange orbits and complementary subalgebras, one can check parts {\sf (ii)} and 
{\sf (iii)} in Corollary~\ref{cor:sledstvie}. Moreover, for $\sln$, there are other methods for computing 
the index of seaweeds. We say more about this in Section~\ref{sect:problems}.

\begin{rmk}    \label{rem:parab-approach}
The ``parabolic'' approach to strangeness in Theorem~\ref{thm:2-col-strange}  
works in many other cases. For instance, one can provide complementary parabolic subalgebras for the orbits $\co(3,2,1)$ 
and $\co(4,1,1)$ in $\sel_6$. Combining this with results of Section~\ref{sect:strange} 
and~\ref{sect:str-glv}, we obtain a curious corollary that {\bf all} orbits in $\mathfrak{sl}_n$, $n\le 5$ are strange; while for $n=6$, the only non-strange nilpotent orbit is $\co(5,1)$. 
\end{rmk}

\subsection{Partitions with three parts} Here we describe a family of partitions with three parts that give rise to strange 
orbits with parabolic complements.
Our formulae and proofs depend on the resudue of $n$ modulo $3$.

\begin{thm}        \label{thm:3-column}  
For any $n\ge 4$, there are the following strange orbits:
\begin{itemize}
\item[{\sf (i)}] If $n=3m$, then $\co(m{+}2,m{-}1,m{-}1)$ is strange and $\p_{\ov{1,3}}$ is a complementary subalgebra;
\item[{\sf (ii)}] If $n=3m{+}1$, then $\co(m{+}3,m{-}1,m{-}1)$ is strange and $\p_{\ov{3}}$ is a complementary subalgebra;
\item[{\sf (iii)}] If $n=3m{+}2$, then $\co(m{+}3,m,m{-}1)$ is strange and $\p_{\ov{3}}$ is a complementary subalgebra.
\end{itemize}
\end{thm}
\begin{proof}
First of all, one can check that $\p$ is Frobenius and $\dim\co=\dim\p$ in all three cases. Then, as in 
Theorem~\ref{thm:2-col-strange}, we choose $x\in\co$ in the Jordan normal form, where the consecutive 
sizes of blocks correspond to the parts indicated above. For instance, the first block in {\sf (i)} is of size 
$m+2$ and its Jordan basis is denoted by $(v_1,\dots,v_{m+2})$, etc. The main difficulty is to guess how 
to choose a flag $\eus F$ of prescribed shape such that $(\sln)^x\cap{\sf Stab}(\eus F)=\{0\}$. 

For {\sf (i)}, $\eus F=\{\gV_1\subset\gV_3\}$ and we take
$\gV_1=\lg v_{m+2}\rg$,  $\gV_3=\lg v_{m+2}, v_{m+1}{+}v_{2m+1}, v_m{+}v_{3m}\rg$.

For {\sf (ii)}, $\eus F=\{\gV_3\}$ and we take
$\gV_3=\lg v_{m+3}, v_{m+2}{+}v_{2m+2}, v_{m+1}{+}v_{3m+1}\rg$.

{\sf (iii)} Here again $\eus F=\{\gV_3\}$, and we take
$\gV_3=\lg v_{m+3}, v_{m+1}{+}v_{2m+3}, v_{m+2}{+}v_{3m+2}\rg$.
\\[.5ex]
To give a flavour of computations and avoid cumbersome notation, we look at case {\sf (i)} with $m=3$, 
i.e., $\blb=(5,2,2)$. According to Fig.~\ref{fig:arnold}, a generic element of $(\sel_9)^x$ is of the form
\begin{center}
{\small  $z=
\begin{pmatrix}
a_1 & a_2 &a_3 &a_4 &a_5 & b_1& b_2& c_1 & c_2 \\
0 & a_1 & a_2 &a_3 &a_4 & 0 &b_1& 0 &c_1 \\
0& 0 & a_1 & a_2 &a_3 & 0 &0 &0 & 0 \\
0 &0 &0 & a_1 & a_2 & 0 &0 &0 & 0 \\
0 &0 &0 & 0 &a_1 & 0 &0 &0 & 0 \\
0 &0 &0 & d_1 & d_2 & m_1 & m_2 & n_1 & n_2 \\
0 &0 &0 & 0 & d_1 & 0 & m_1 & 0 & n_1 \\
0 &0 &0 & e_1 & e_2 & p_1 & p_2 & q_1 & q_2 \\
0 &0 &0 & 0 & e_1 & 0 & p_1 & 0 & q_1 \\
\end{pmatrix}$, }
\end{center}
where $5a_1+2m_1+2q_1=0$. Here $\gV_1=\lg v_5\rg$ and $\gV_3=\lg v_5, v_4+v_7, v_3+v_9 \rg$.
Suppose that $z\in(\sel_9)^x$ preserves the flag $\{\gV_1\subset\gV_3\}$. Then

$v_5\stackrel{z}{\mapsto} (a_5,a_4,a_3,a_2,a_1,d_2,d_1,e_2,e_1)\in \gV_1$. Hence 
all these coordinates, except $a_1$, are equal to zero. Next, using these conditions, we must have

$v_4+v_7\stackrel{z}{\mapsto} (b_2,b_1,0,a_1,0, m_2,m_1,p_2,p_1)\in \gV_3$ \ and 

$v_3+v_9\stackrel{z}{\mapsto} (c_2,c_1,a_1,0, 0, n_2,n_1,q_2,q_1)\in \gV_3$. 

\noindent
This implies that $a_1=m_1=q_1$ and all other matrix entries of $z$ equal zero. Because of the trace 
condition, we conclude that if $z\in (\sel_9)^x$ preserves the flag, then $z=0$.
\end{proof}

Results Sections~\ref{sect:str-glv} and \ref{sect:two-col} suggest that strange orbits in $\sln$ are perhaps not so sparse, at least for small 
$n$. There are also other constructions (partly conjectural) that provide strange nilpotent orbits with 
parabolic complements. But my feeling is that, for large $n$, most of the nilpotent orbits in $\sln$ are not 
strange.

\section{Further perspectives and conjectures}   
\label{sect:problems}

\noindent
\subsection{Towards a classification of strange orbits and Frobenius subalgebras}
\label{subs:towards}
Results of Sections~\ref{sect:str-glv} and~\ref{sect:two-col} show that there are many strange 
non-spherical orbits in $\sln$. Moreover, our computations demonstrate that, for sufficiently large 
nilpotent $SL_n$-orbits, their strangeness is often related to the fact that they possess 
complementary {\bf parabolic} subalgebras. However, attempting to apply this idea to other simple Lie 
algebras, we face a disappointing fact that if $\ind\be=0$, then there are no other Frobenius parabolic 
subalgebras in $\g$. This follows from the formula for the index of {\it seaweeds\/} in semisimple Lie 
algebras, which is conjectured in~\cite[4.7]{ty} and proved in~\cite[Sect.\,8]{jos}. We call it the 
{\it Tauvel--Yu--Joseph ({=}TYJ) formula}. By definition, $\es$ is a seaweed if there are parabolic 
subalgebras $\p,\tilde\p\subset\g$ such that $\p+\tilde\p=\g$ and $\p\cap\tilde\p=\es$~\cite[Sect.\,2]{mmj}.
The TYJ formula exploits the {\it Kostant cascade\/} for Levi subalgebras of $\p$ and $\tilde\p$. 
If $\tilde\p=\g$, then $\es=\p$, i.e., any parabolic subalgebra is a seaweed.
Let $\el$ be a Levi subalgebra of $\p$. If $\ind\be=0$ and $\eus K(\el)$ 
is the Kostant cascade for $\el$, then the TYJ formula says that $\ind\p=\#\eus K(\el)$.
Thus, if $\p\ne\be$, then $\eus K(\el)\ne\varnothing$ and $\ind\p>0$.
Hence non-solvable Frobenius parabolics may exist only if $\ind\be>0$, i.e., for 
$\GR{A}{m}\ (m{\ge} 2), \GR{D}{2m+1}\ (m{\ge} 2), \GR{E}{6}$. 

For $\sln$, the index of seaweeds can be computed via {\it meander graphs\/} of Dergachev--Kirillov,
which is quite easy if $n$ is not too large. There is also an inductive formula that always works for $\sln$ 
and $\spn$~\cite{mmj}. For instance, a seaweed in $\sln$ (or $\gln$) is determined by two compositions of 
$n$, and the induction step reduces the problem to a seaweed in $\sel_{m}$ with $m<n$, i.e., to a pair of 
compositions of $m$, see~\cite[Theorem\,4.2]{mmj}.

Recall that $\ind\be(\sln)=\lfloor(n{-}1)/2\rfloor$, $\ind\be(\mathfrak{so}_{4m+2})=1$, and 
$\ind\be(\GR{E}{6})=2$. There are only few standard Frobenius parabolics in $\GR{D}{2m+1}$ and 
$\GR{E}{6}$, and it is not hard to list all of them. But we unable to detect any strange orbits in these two 
cases so far, not to mention strange orbits with parabolic complements. For $\sln$, the Frobenius 
parabolics are in abundance and their explicit description is not known yet.   

It is easily seen that if $\p=\el\oplus\p^{\sf nil}$ is a Frobenius parabolic subalgebra of minimal dimension, 
then $\dim\p=\dim\be+\ind\be$ and $[\el,\el]$ is a sum of $\ind\be$ copies of $\tri$. Using the TYJ  
formula, one easily computes that the set of dimensions of Frobenius parabolics is
\begin{itemize}
\item $\{\dim\be+1, \dim\be+3\}$ for $\mathfrak{so}_{4m+2}$, where $\dim\be=(2m+1)^2$;
\item $\{44,\,46,\,48,\, 52\}$ for $\GR{E}{6}$.
\end{itemize}
If $\p\subset\sln$ is a Frobenius parabolic subalgebra, then ($\dim\p$ is even and)
\[
        \frac{n(n+1)}{2}-1+\lfloor\frac{n{-}1}{2}\rfloor\le \dim\p\le n^2-n .
\]
But not all even integers in this interval occur as dimensions of Frobenius parabolics. A gap in this interval 
appears first for $n=6$, where the set of dimensions is $\{22,24,26,30\}$. For $\sel_9$, the interval is $[48,72]$ and the set of gaps is $\{58, 62, 64,68,70\}$. That is, the larger $n$, the more gaps we get.
Recall that the number $\eus M_{\sf Fr}$ is defined in Section~\ref{subs:numero}.

\begin{conj}    \label{conj:b-Frob}   
Let $\g$ be a simple Lie algebra. Then
\begin{itemize}
\item[\sf (i)] \  the value $\eus M_{\sf Fr}$ is always attained on Frobenius {\bf parabolic} subalgebras. More precisely, if $\ind\h=0$, then there is a Frobenius parabolic subalgebra $\p$ such that $\h\subset\p$.
\item[\sf (ii)] \ If\/ $\ind\be=0$, then 
\begin{itemize}
\item $\dim\h\le \dim\be$ for any Frobenius subalgebra $\h\subset\g$. 
\item only the spherical orbits are strange. 
\end{itemize}
\end{itemize}
\end{conj}

\begin{ex}    \label{ex:g_2}  \leavevmode\par
(a) \ Conjecture~\ref{conj:b-Frob}{\sf (i)} is true for $\sln$, because $\dim\p_{\ov{1}}$ is maximal among {\bf all} subalgebras; \\ \indent
(b) \ The small rank cases with $\ind\be=0$ can be verified by hands. \\
{\bf --} \ For $\g=\GR{G}{2}$, the maximal subalgebras are:
\\ \indent
\textbullet \ \ Two maximal parabolic subalgebras $\p_i$ ($i=1,2$), where $\dim\p_i=9$ \& $\ind\p_i=1$;
\\ \indent
\textbullet \  \ semisimple: $\tri$, $\tri\dotplus\tri$, and $\mathfrak{sl}_3$,  where $\ind\es=1$ or $2$. 
\\ 
Hence $\be=\be(\GR{G}{2})$ is a Frobenius subalgebra of maximal dimension, $\dim\be=8$. Two spherical nilpotent orbits in $\GR{G}{2}$ are of dimension 6 and 8.

\noindent
{\bf --} \ For $\GR{B}{2}=\GR{C}{2}$, one has $\dim\be=6$, and the argument is similar. For 
$\GR{B}{3}$ and $\GR{C}{3}$, the argument is more involved, but these cases are still manageable.
\end{ex}

Conjecture~\ref{conj:b-Frob}{\sf (ii)} would imply that if $\ind\be=0$, then $\gM_{\sf sph}=\gM_{\sf str}= \gM_{\sf Fr}=\dim\be$.

Let $\z(\g^x)$ denote the centre of $\g^x$. It is also the centraliser of $\g^x$ in $\g$.

\begin{conj}          \label{conj:z^2(e)}
If\/ $G{\cdot}e\subset\N$ is strange and $e'\in\z(\g^e)$, then $G{\cdot}e'$ is strange, too.
\end{conj}

If $G{\cdot}e$ is spherical, then $\z(\g^e)=\bbk e$. Therefore, this conjecture can only be helpful for 
non-spherical orbits. If Conjecture~\ref{conj:b-Frob} is true, then Conjecture~\ref{conj:z^2(e)} is 
needed mainly for $\sln$. For $e\in\N(\sln)$, the centre  $\z((\sln)^e)$ is the linear span
of $\{e^k\mid k\ge 1\}$. Hence specialising Conjecture~\ref{conj:z^2(e)} to $\sln$, we get the following 
conjectural generalisation of Theorem~\ref{thm:strange-gln}.

\begin{conj}
If\/ $\co=SL_n{\cdot}e\subset\N$ is strange, then so is $\co^{\lg k\rg}:=SL_n{\cdot}e^k$ $(k\ge1)$.
\end{conj}

For any $n\ge 3$, there is a unique orbit $\co(\blb)\subset\sln$ such that $\dim\co(\blb)=\dim\be+\ind\be$,
$\lb_1=3$, and $\lb_2<3$. Namely,
\begin{center}
\begin{tabular}{crlll}
(a) & $\sel_{2k}$: &  $\blb=(3,2^{k-2},1)$, & $\ind\be=k{-}1$, & $\dim\co(\blb)=2k^2+2k-2$; \\
(b) & $\sel_{2k+1}$: &  $\blb=(3,2^{k-1})$, & $\ind\be=k$, & $\dim\co(\blb)=2k^2+4k$. \\
\end{tabular}
\end{center}
Note that since $\lb_1>2$, this orbit is not spherical. Sometimes there can be other nilpotent orbits of
dimension $\dim\be+\ind\be$, e.g. $\co(4,1^3)$.
\begin{conj}    \label{conj:min-frob}
This orbit $\co(\blb)\subset\sln$ is strange and a complementary subalgebra is a Frobenius parabolic 
subalgebra of minimal dimension. 
\end{conj}

This has been verified for $n\le 7$, i.e., if the number of parts of $\blb$ is at most 3. 
To obtain a minimal Frobenius parabolic $\p=\p_A$ in $\sln$, one may take 
\[
   A=\begin{cases} 
   \{2,4,\dots,2k\}  &  \text{ if } n=2k+1, \\
   \{2,4,\dots,2p-2,2p+1,\dots,4p-1\}  &  \text{ if } n=4p, \\
   \{2,4,\dots,2p,2p+3,\dots,4p+1\}  &  \text{ if } n=4p+2. \\ 
   \end{cases}
\]
For all cases, $\#A=\ind\be$ and $[\el_A,\el_A]\simeq \tri\dotplus\ldots\dotplus\tri$, wih $\ind\be$ summands.

\begin{rmk}     \label{rem:(b)-strange}
In case (b), we have $\co(3,2^{k-1})=\co_{\sf pr}^{\lg k\rg}$, i.e., this orbit is strange by 
Theorem~\ref{thm:strange-gln}. However, a complementary subalgebra given therein is not parabolic. 
Therefore, in this case, the point of Conjecture~\ref{conj:min-frob} is to prove that a minimal Frobenius 
parabolic subalgebra is also complementary.
\end{rmk}

\subsection{Strange orbits and compatible Lie brackets}
\label{subs:strange-PC}
For a strange orbit $\co=G{\cdot}e\subset\g\simeq\g^*$, the splitting  $\g=\g^e\oplus\h$ yields a pencil
of compatible Lie brackets in the vector space $\g$, hence a pencil of compatible Poisson brackets in the
symmetric algebra $\cs(\g)$. Generic elements of this pencil provide Lie brackets
isomorphic to the initial semisimple Lie algebra $\g$, whereas degenerate Lie brackets correspond to
the semi-direct products $\g_{(0)}=\g^e\ltimes\h$ and $\g_{(\infty)}=\h\ltimes\g^e$, see~\cite[Section\,3]{bn}.
Then the Lenard--Magri scheme provides the Poisson--commutative subalgebra 
$\gZ_{\lg\g^e,\h\rg}$ of the Poisson algebra $\cs(\g)$.

If $\co$ and $G/H$ are spherical $G$-varieties, then 
$\ind\g_{(0)}=\ind\g_{(\infty)}=\ind\g$~\cite[Sect.\,2]{bn},\cite[Theorem\,2]{T}. In this case, the algebra 
$\gZ_{\lg\g^e,\h\rg}$ may have the maximal possible transcendence degree 
$\frac{1}{2}(\dim\g+\rk\g)=\dim\be$. This possibility realises if $\co$ is the maximal spherical nilpotent 
orbit. Then $\dim\co=\dim\be-\ind\be$ and $\h$ is a solvable algebra that contains $[\be,\be]$. In 
particular, if $\ind\be=0$, then $\h=\be$ and it is an interesting problem to describe $\gZ_{\lg\g^e,\be\rg}$.
It is likely that in this case $\trdeg \gZ_{\lg\g^e,\h\rg}=\dim\be$ and $\gZ_{\lg\g^e,\h\rg}$ is a polynomial 
ring. 

Yet another  possibility is to consider the {\it transverse Poisson structure\/} on the Slodowy slice
$e+\g^f$ for an $\tri$-triple $\{e,h,f\}$, i.e., for the orbit $\co=G{\cdot}e\subset\N$. This Poisson structure (bracket) is polynomial~\cite{cr} and its
linear part is nothing but the Lie-Poisson bracket for the Lie algebra $\g^f$.
If the orbit $\co$ is strange, then this Poisson structure splits into the linear and quadratic parts. This yields a pair of compatible Poisson brackets in $\cs(\g^f)$. The study of the corresponding 
Poisson-commutative subalgebras and integrable systems looks like a rather promising enterprise.

\appendix
\section{Actions of reductive subgroups on the nilpotent cone}  
\label{app}

\noindent
Here we prove the result that was used in Proposition~\ref{prop:nilp}.
In this section, $G$ is a connected semisimple group, $\N=\N(\g)$ is the nilpotent cone in $\g$, and $H$ 
is a connected reductive subgroup of $G$. We assume that  $H$ contains no infinite normal subgroups of 
$G$, i.e., $H$ is essentially proper in $G$.

\begin{lm}   \label{lm:app1}
The action $(H:\g)$ is stable and locally free. 
\end{lm}
\begin{proof}
Since the adjoint action $(G:\g)$ is stable and $H$ is reductive, the stability of $(H:\g)$ follows 
from~\cite[Theorem\,4]{vi00}. Generic $G$-orbits in $\g$ are semisimple and isomorphic to $G/T$, and 
the $H$-action on $G/T$ is locally free (cf. Lemma~\ref{lm:ss}).
\end{proof}

Our goal is to carry this result over the nilpotent cone $\N\subset\g$. Fix a triangular decomposition of 
$\g$ and consider the corresponding simple roots $\Pi$ and dominant weights $\mathfrak X_+$ in $\te^*$, 
where $\te=\Lie T$. Write $\gV_\lb$ for the simple $G$-module with highest weight $\lb\in\fX_+$. The 
highest weight of the dual  $G$-module is denoted by $\lb^*$, i.e.,
$(\gV_\lb)^*=\gV_{\lb^*}$. For a $G$-variety $X$, consider
the isotypic decomposition of the algebra of regular functions~\cite[II.3.1]{kr84}: 
\[
     \bbk[X]=\bigoplus_{\lb\in\fX_+}\bbk[X]_{(\lb)} 
\]
and set $\Gamma(X):=\{\lb\mid \bbk[X]_{(\lb)}\ne 0\}$. If $X$ is irreducible, then $\Gamma(X)$ is a 
monoid. If $X$ has a dense $G$-orbit,  then $\dim \bbk[X]_{(\lb)}< \infty$ for all $\lb\in\Gamma(X)$
and $m_\lb(X)=\dim\bbk[X]_{(\lb)}/\dim\gV_\lb$ is called the {\it multiplicity\/} of $\lb$ in $\bbk[X]$. Recall 
some standard facts on multiplicities.
\begin{itemize}
\item For a homogeneous space $G/F$, one has 
$m_\lb(G/F)=\dim(\gV_{\lb^*})^F$ for any $\lb\in\fX_+$ (Frobenius reciprocity).
\item $\Gamma(\N)=\fX_+\cap \eus Q$, where $\eus Q$ is the {\it root lattice}, and 
$m_\lb(\N)=m_\lb^0$, the zero weight multiplicity in $\gV_\lb$~\cite{ko63}.
\end{itemize}

\begin{lm}   \label{lm:app2}
$\bbk[\N]^H\ne \bbk$, i.e., $H$ has non-trivial invariants in $\bbk[\N]$.
\end{lm}
\begin{proof}
Since $H$ is reductive, the monoid $\Gamma(G/H)$ is self-dual, i.e., $\Gamma(G/H)=\Gamma(G/H)^*$.
Therefore,  $\lb\in\Gamma(G/H)$ if and only if $(\gV_\lb)^H\ne 0$. Given $\lb\in\Gamma(G/H)$, 
there is $a\in\BN$ such that $a\lb\in\eus Q$. Then $\gV_{a\lb}\subset\bbk[\N]$ and 
$0\ne (\gV_{a\lb})^H\subset\bbk[\N]^H$.
\end{proof}

Since $\bbk[\N]^H\ne \bbk$, we have $\N\md H\ne \{pt\}$ and there are non-trivial closed $H$-orbits in 
$\N$. 
\begin{lm}          \label{lm:app3}
If there is a closed $H$-orbit contained in $\N_{\sf reg}$, then the action $(H:\N)$ is stable and locally free.
\end{lm}
\begin{proof}
If $H{\cdot}x\subset \N_{\sf reg}$ is closed, then $H^x$ is reductive. On the other hand,
$H^x=H\cap G^x$ and $(G^x)^0$ is unipotent. Hence $H^x$ must be finite. Thus, $H{\cdot}x$ is a closed 
orbit of maximal dimension. This implies stability, see e.g.~\cite[Theorem\,1]{vi00}.
\end{proof}

Recall that $\N_{\sf reg}=\co_{\sf pr}$ is the dense $G$-orbit in $\N$ and 
$\N_{\sf sing}:=\N\setminus\N_{\sf reg}$ is the singular locus of $\N$, which is of codimension 2 in $\N$. 
If $G$ is simple, then $\N_{\sf sing}$ is irreducible and hence it is the closure of a  $G$-orbit. In general, if 
$\g=\g_1\dotplus\ldots\dotplus\g_t$ is the sum of $t$ simple ideals, then $\N_{\sf sing}$ has $t$ irreducible
components. Namely, in an obvious notation, if 
\[
   Y_i=\N_{\sf sing}(\g_i) \times \prod_{j\ne i}\N(\g_j)\subset \prod_{i=1}^t\N(\g_i)=\N ,
\] 
then $\N_{\sf sing}(\g)=\bigcup_{i=1}^t Y_i$.

\begin{lm}          \label{lm:app4}
There are closed $H$-orbits that are not contained in $\N_{\sf sing}$.
\end{lm}
\begin{proof}
Assume the contrary. Then the action $(H:\N)$ is not stable and by~\cite[n.\,3]{vi00} there is the 
{\bf maximal} irreducible subvariety $\N_{st}\subset \N$ such that $(H:\N_{st})$ is stable and 
$\bbk[\N]^H=\bbk[\N_{st}]^H$. Actually, $\N_{st}$ is the closure of the union of all closed $H$-orbits. Hence 
$\N_{st}\subset\N_{\sf sing}$. If $\g=\g_1\dotplus\ldots\dotplus\g_t$, then there is at least one index
$i\in[1,t]$ such that $\N_{st}\subset Y_i$.  

Under our assumptions, we have $\bbk[\N]^H=\bbk[Y_i]^H$. Hence the ideal of $Y_i$ in $\bbk[\N]$ does 
not contains $H$-invariants. For the dominant weights of $\g$, we have
\[
     \fX_+(\g)=\fX_+(\g_1)\times\ldots \times\fX_+(\g_t)
\]
and hence $\lb=(\lb_1,\dots,\lb_t)$ with $\lb_j\in \fX_+(\g_j)$. Since $\g_i\not\subset\h$, one can find
$\lb\in\Gamma(G/H)\cap \eus Q$ such that $\lb_i\ne 0$. Let us compare the multiplicities of 
$\gV_\lb=\gV_{\lb_1}\otimes\ldots\otimes \gV_{\lb_t}$
in $\bbk[\N]$ and $\bbk[Y_i]$. The structure of $\N$ and $Y_i$ shows that
\[
  m_\lb(\N)- m_\lb(Y_i)=m_{\lb_i}(\N(\g_i))- m_{\lb_i}(\N_{\sf sing}(\g_i)) .
\]
Let $\ap$ be a short simple root of $\g_i$. (In the simply-laced case, all simple roots are assumed to be 
short.)  For the simple Lie algebra $\g_i$, it follows from~\cite[Corollary\,4.7]{br93} that
\[
  m_{\lb_i}(\N(\g_i))- m_{\lb_i}(\N_{\sf sing}(\g_i))=m_{\lb_i}^{\ap} ,
\]
the multiplicity of weight $\ap$ in the $\g_i$-module $\gV_{\lb_i}$. Thus, the multiplicity of $\gV_\lb$ in 
the ideal of the subvariety $Y_i\subset \N$ equals $m_{\lb_i}^{\ap}$. If $\lb\in \Gamma(G/H)\cap \eus Q$ 
is sufficiently large, then so is $\lb_i$ and hence $m_{\lb_i}^{\ap}>0$, which yields an $H$-invariant in the ideal of $Y_i$. This contradiction shows that there must be closed $H$-orbits in $\N_{\sf reg}$.
\end{proof}
Combining Lemmata~\ref{lm:app3} and \ref{lm:app4}, we obtain
\begin{thm}     \label{thm:app}
The action of $H$ on $\N$ is stable and locally free.
\end{thm}

\noindent
{\bf Data availability and conflict of interest statement.}

This article has no associated data. There is no conflict of interest. 



\begin{thebibliography}{Pa95}

\bibitem{ar71}      {\sc V.I.\,Arnol'd}.
On matrices depending on parameters, 
{\it Russ. Math. Surv.}, {\bf 26}\,(1971), no. 2, 29--43.

\bibitem{bo81}            {\sc W.\,Borho}.
\"Uber Schichten halbeinfacher Lie-Algebren, 
{\it Invent. Math.}, {\bf 65}\,(1981), 283--317.

\bibitem{bk79}           {\sc W.\,Borho} and {\sc H.\,Kraft}.
\"Uber Bahnen und deren Deformationen bei Aktionen reduktiver Gruppen, 
{\it Comment. Math. Helv.},  {\bf 54}\,(1979), 61--104.

\bibitem{br93}          {\sc A.\,Broer}.
Line bundles on the cotangent bundle of the flag variety, 
{\it Invent. Math.}, {\bf 113}\,(1993), 1--20.

\bibitem{CM}  {\sc D.H.\,Collingwood} and {\sc W.M.\,McGovern}. "{\it Nilpotent orbits in semisimple
  Lie algebras}", New York: Van Nostrand Reinhold, 1993.

\bibitem{cm10}      {\sc J.-Y.~Charbonnel} and {\sc A.~Moreau}.
The index of centralizers of elements of reductive Lie algebras,
{\it  Doc. Math.}, {\bf 15}\,(2010), 387--421.

\bibitem{cr}       {\sc R.\,Cushman} and {\sc M.\,Roberts}, 
Poisson structures transverse to coadjoint orbits, 
{\it Bull. Sci. Math.} {\bf 126}\,(2002), no.\,7, 525--534.

\bibitem{dixm}      {\sc J.\,Dixmier}.
Polarisations dans les alg\`ebres de Lie semi-simples complexes,
{\it Bull. Sci. Math.}, {\bf 99}\,(1975), no.\,1, 45--63.

\bibitem{deG-AG}       {\sc W.\,de Graaf} and {\sc A.G.\,Elashvili}.
Induced nilpotent orbits of the simple Lie algebras of exceptional type,
{\it Georgian Math. J.}, {\bf 16}\,(2009), no.\,2, 257--278.

\bibitem{ag85}           {\sc A.G.\,Elashvili}.
Sheets of the simple Lie algebras of exceptional type, in 
"{\it Issledovaniya po algebre}'' (=\,{\it Studies in algebra}), p.171--194. Tbilisi Univ. Press, 1985 (Russian).

\bibitem{GS}             {\sc M.I.\,Gekhtman} and {\sc A.\,Stolin}.
Orbits of the coadjoint representation and Yang-Baxter equation. 
Y.\,Fong et al.\,(eds.), ``{\it First international Tainan--Moscow algebra workshop}'', Berlin: de Gruyter. 207--223 (1996).

\bibitem{jos}               {\sc A.\,Joseph}.
On semi-invariants and index for biparabolic (seaweed) algebras I, 
{\it J. Algebra}, {\bf 305}\,(2006), 487--515.

\bibitem{K51}              {\sc F.I.\,Karpelevich}. 
On nonsemisimple maximal subalgebras of semisimple Lie algebras, 
{\it Doklady Akad. Nauk SSSR},
{\bf 76}\,(1951), no.\,6, 775--778 (Russian).

\bibitem{kats82}          {\sc P.\,Katsylo}.
Sections of sheets in a reductive algebraic Lie algebra (Russian). 
{\it Izv. Akad. Nauk SSSR Ser. Mat.} {\bf 46}\,(1982), no.\,3, 477--486. 
(English translation: {\it Math. USSR-Izv.} {\bf 20}\,(1982), no.\,3, 449--458.)

\bibitem{kemp}          {\sc G.\,Kempken}. 
Induced conjugacy classes in classical Lie-algebras,
{\it Abh. Math. Sem. Univ. Hamburg}, {\bf 53}\,(1983), 53--83.

\bibitem{ko63}     {\sc B.\,Kostant}. 
Lie group representations on polynomial rings, 
{\it  Amer. J. Math.}, {\bf 85}\,(1963), 327--404.

\bibitem{kr84}    {\sc H.\,Kraft}. 
``{\it Geometrische Methoden in der Invariantentheorie\/}'', Aspekte der Mathematik {\bf D1}, 
Braunschweig: Vieweg \& Sohn, 1984.

\bibitem{oh}      {\sc Y.-G.\,Oh}.
Some remarks on the transverse Poisson structures of coadjoint orbits,
{\it Lett. Math. Phys.}, {\bf 12}\,(1986), no.\,2, 87--91. 

\bibitem{tg97}      {\sc D.\,Panyushev}. 
Cones of highest weight vectors, weight polytopes, and Lusztig's $q$-analog, 
{\it Transformation Groups}, {\bf 2}\,(1997), 91--115.

\bibitem{nilp-tori}  {\sc D.\,Panyushev}. 
Actions of ``nilpotent tori'' on $G$-varieties,
{\it Indag. Math.}, {\bf 10}(4)~(1999), 565--579.

\bibitem{aif99}      {\sc D.\,Panyushev}. 
On spherical nilpotent orbits and beyond, 
{\it  Ann. Inst. Fourier}, {\bf 49}\,(1999), 1453--1476.

\bibitem{disser}     {\sc D.\,Panyushev}. 
Complexity and rank of actions in invariant theory, 
{\it  J. Math. Sci.} (New York) {\bf 95}\,(1999), 1925--1985.

\bibitem{mmj}         {\sc D.\,Panyushev}.
 Inductive formulas for the index of seaweed Lie algebras,
{\it Moscow Math. J.}, {\bf 1}\,(2001), 221--241.

\bibitem{p03}         {\sc D.\,Panyushev}.
The index of a Lie algebra, the centraliser of a nilpotent element, and the normaliser of the centraliser, 
{\it Math. Proc. Camb. Phil. Soc.}, {\bf 134}, Part\,1 (2003), 41--59.

\bibitem{py06}       {\sc D.\,Panyushev} and {\sc O.\,Yakimova}.
The index of representations associated with stabilisers, 
{\it J. Algebra}, {\bf 302}\,(2006), 280--304.

\bibitem{bn}           {\sc D.\,Panyushev} and {\sc O.\,Yakimova}. 
Compatible Poisson brackets associated with splittings and Poisson commutative 
subalgebras of $\gS(\g)$, 
{\it J. London Math. Soc.}, {\bf 103}\,(2021), no.\,4, 1577--1595.  

\bibitem{s05}            {\sc H.\,Sabourin}.  
Sur la structure transverse \`{a} une orbite nilpotente adjointe,
{\it Canad. J. Math.}, {\bf 57}\,(2005), no.\,4, 750--770.

\bibitem{slod}       {\sc P.\,Slodowy}.
{\it "Simple singularities and simple algebraic groups"},
(Lecture Notes in Math.  {\bf 815}),  Berlin: Springer, 1980.

\bibitem{spalt}      {\sc N.\,Spaltenstein}. 
{\it ``Classes unipotentes et sous-groupes de Borel"}
(Lecture Notes in Math. {\bf 946}),  Berlin 
Heidelberg New York: Springer 1982.

\bibitem{ty}           {\sc P.\,Tauvel} and {\sc R.W.T.\,Yu}.
Sur l'indice de certaines alg\`ebres de Lie,
{\it Ann. Inst. Fourier} (Grenoble) {\bf 54}\,(2004), no.\,6, 1793--1810.

\bibitem{T}              {\sc D.A.\,Timashev}. 
Index of In\"on\"u--Wigner contractions of semisimple Lie algebras,
{\it Russ. J. Math. Phys.}, {\bf 32}\,(2025), 189--195. 

\bibitem{vi00}        {\sc E.B.\,Vinberg}.
On stability of actions of reductive algebraic groups,
Y.\,Fong et al. (eds.), ``{\it Lie algebras, rings and related topics}''. Papers of the 2nd Tainan-Moscow international algebra workshop '97, Tainan, Taiwan, 1997. Hong Kong: Springer,  188--202 (2000).

\bibitem{vp}          {\sc E.B.\,Vinberg} and {\sc V.L.\,Popov}. 
{\it Invariant theory}, in: Sovremennye problemy matematiki. Fundam. napravl., t.\,{\bf 55}, p.\,137--309, 
Moskva:~VINITI, 1989 (Russian). English translation in: {\it  Algebraic Geometry IV} (Encyclopaedia Math. 
Sci., vol.~55, p.123--284) Berlin Heidelberg New York: Springer 1994.

\bibitem{yu08}       {\sc R.W.T.~Yu}.
On the sum of the index of a parabolic subalgebra and of its nilpotent radical,
{\it Proc. Amer. Math. Soc.}, {\bf 136}\,(2008), no.\,5, 1515--1522.

\end{thebibliography}
\end{document}